\documentclass[a4paper, english, 12pt]{article}
\usepackage{babel}
\usepackage{amsmath,amsfonts,amssymb,amsthm}
\usepackage{fullpage}
\usepackage{bbm}
\usepackage[latin1]{inputenc}
\usepackage{enumerate}

\usepackage{color}
\definecolor{linkcol}{rgb}{0,0,0.4} 
\definecolor{citecol}{rgb}{0.5,0,0}
\definecolor{urlcol}{rgb}{0,0,0.75}

\usepackage{hyperref}

\hypersetup
{
bookmarksopen=true,
pdftitle=Random measurable sets and covariogram realisability problems,
pdfauthor=Bruno Galerne and Raphael Lachieze-Rey,
pdfsubject=Random measurable sets and covariogram realisability problems, 
pdfhighlight=/O, 
colorlinks=true, 
pdfpagemode=UseOutlines, 
pdfpagelayout=SinglePage, 
pdffitwindow=true, 
linkcolor=linkcol, 
citecolor=citecol, 
urlcolor=urlcol 
}

\theoremstyle{definition}
\newtheorem{definition}{Definition}
\newtheorem{remark}{Remark}

\newtheoremstyle{mytheorem}{0.5cm}{0.2cm}{\slshape}{ }{\bfseries}{.}{ }{}
\theoremstyle{mytheorem}
\newtheorem{theorem}{Theorem}
\newtheorem{lemma}{Lemma}
\newtheorem{proposition}{Proposition}

\newtheorem{corollary}{Corollary}

\begin{document}

\title{Random Measurable Sets and Covariogram Realisability Problems\thanks{This is a preprint version of the paper that has been accepted in \textit{Advances in Applied Probability}.}}

\author{Bruno Galerne and Rapha\"{e}l Lachi\`{e}ze-Rey\thanks{Laboratoire MAP5 (UMR CNRS 8145), Université Paris Descartes, Sorbonne Paris Cité ; bruno.galerne@parisdescartes.fr ; raphael.lachieze-rey@parisdescartes.fr ;} }

\maketitle

\begin{abstract}
We provide a characterization of realisable  set covariograms, bringing a rigorous yet abstract solution to the  $S_2$ problem in materials science. Our method is based   on the covariogram functional for random measurable sets (RAMS) and on a result about the representation of positive operators on a non-compact space. 
RAMS are an alternative to the classical   random closed sets in stochastic geometry and geostatistics,  they provide a weaker framework allowing to manipulate more irregular functionals, such as the perimeter.  We therefore use the illustration provided by the $S_{2}$ problem to advocate the use of RAMS for solving theoretical problems of geometric nature.   Along the way, we extend the theory of random measurable sets, and in particular the local approximation of the perimeter by local covariograms.
\end{abstract}


{\it MSC2010 subject classification:} Primary 60D05; Secondary 28C05

\renewcommand{\L}{\mathcal{L}}
\newcommand{\E}{\mathbf{E}}
\newcommand{\R}{\mathbb{R}}
\newcommand{\C}{\mathbb{C}}
\newcommand{\Z}{\mathbb{Z}}
\newcommand{\N}{\mathbb{N}}
\newcommand{\Q}{\mathbb{Q}}
\newcommand{\EE}{\mathcal{X}}
\newcommand{\Prob}[1]{\mathbf{P}\{#1\}}
\newcommand{\one}{\mathbbm{1}}
\newcommand{\Ind}{\mathbbm{1}} 
\newcommand{\onev}{{\mathbf{1}}}
\newcommand{\card}{card}
\newcommand{\dmet}{\mathsf{d}}
\renewcommand{\P}{\mathbf{P}}
\newcommand{\PerB}{\Per_{\mathbf{B}}}
\newcommand{\LipB}{\Lip_{\mathbf{B}}}
\newcommand{\sB}{\mathcal{B}}
\newcommand{\sC}{\mathcal{C}}
\newcommand{\sK}{\mathcal{K}}
\newcommand{\sF}{\mathcal{F}}
\newcommand{\F}{\mathcal{F}}
\newcommand{\sM}{\mathcal{M}}
\newcommand{\M}{\mathcal{M}} 
\newcommand{\sX}{\mathbb{X}}    
\newcommand{\sXX}{\sX\times\sX}
\newcommand{\sN}{\mathcal{N}}
\newcommand{\sYY}{\mathbb{Y}}
\newcommand{\sR}{\mathcal{R}}
\newcommand{\sP}{\mathcal{P}}
\newcommand{\loc}{{\operatorname{loc}}}
\newcommand{\Lip}{\operatorname{Lip}}
\newcommand{\W}{\mathcal{W}} 

 \newcommand{\dom}{\operatorname{dom}}
\renewcommand{\H}{\mathcal{H}}
\newcommand{\B}{\mathcal{B}}
\newcommand{\K}{\mathcal{K}}

\newcommand{\Cc}{\mathsf{C}_{\mathrm{o}}}
\newcommand{\fnc}{\Phi}
\newcommand{\gspace}{\mathsf{G}}
\newcommand{\espace}{\mathsf{E}}
\newcommand{\Vspace}{\mathsf{V}}
\newcommand{\Hspace}{\mathsf{H}}

\newcommand{\fF}{\mathfrak{F}}
\newcommand{\leb}{\mathcal{L}^d}
\newcommand{\eps}{\varepsilon}
\renewcommand{\phi}{\varphi}
\renewcommand{\kappa}{\varkappa}

\newcommand{\SymD}{\Delta }

\newcommand{\supp}{{\rm{supp}}}
\newcommand{\Sd}{\mathcal{S}^{d-1}}
\newcommand{\Per}{{\operatorname{Per}}}

\newcommand{\thf}[1][1]{{\textstyle\frac{#1}{2}}}
\newcommand{\pack}[1]{P_{#1}(\sX)}
\newcommand{\packs}[2]{P_{#1}(#2)}
 \newcommand{\eae}{\stackrel{a.e}{=}}
\newcommand{\Var}{{\rm Var}}
\newcommand{\RM}{regularity modulus}

 \section{Framework and main results}
\label{sec:framework}

\subsection{Introduction}

An old and difficult problem in materials science is the $S_{2}$ problem, often posed in the following terms:
Given a real function $S_{2} :\mathbb{R}^{d}\rightarrow [0,1]$, is there a stationary random  set $X\subset \mathbb{R}^{d}$ whose standard two point correlation function is $S_{2}$, that is, such that 
\begin{equation}
\label{eq:S2-definition}
\P(x,y\in X)=S_{2}(x-y),\;x,y\in \mathbb{R}^{d}~?
\end{equation}
The $S_{2}$ problem is a \emph{realisability problem} concerned with the existence of a (translation invariant) probability measure satisfying some prescribed marginal conditions. 

This question is the stationary version of the problem of characterizing functions $S(x,y)$ satisfying 
$$
\begin{aligned}
S(x,y)=\P(x,y\in X)=\E \Ind_{X}(x)\Ind_{X}(y).
\end{aligned}
$$
The right-hand term is the second order moment of the random indicator field $x\mapsto \Ind_{X}(x)$, which justifies the term of \emph{realisability problems}, concerned with the existence of a positive measure satisfying some prescribed moment conditions.

One can see the $S_{2}$ problem as a truncated version of the general \emph{moment problem} that deals with the existence of a process for which all moments are prescribed.
The main difficulty in only considering the moments up to some finite order is that this sequence of moments does not uniquely determine the possible solution.
The appearance of second order realisability problems for random sets goes back to the 1950's, see for instance~\cite{McM55} in the field of telecommunications. 
There are applications in materials science and geostatistics, and marginal problems in general are present under different occurrences in fields as various as quantum mechanics, computer science, or game theory, see the recent work~\cite{FriCha13} and references therein.

Reconstruction of heterogeneous materials from a knowledge of limited microstructural information (a set of lower-order correlation functions) is a crucial issue in many applications. Finding a constructive solution to the realisability problem described above should allow one to test whether an estimated covariance indeed corresponds to a random structure, and propose an adapted reconstruction procedure. 
Studying this problem can serve many other purposes, especially in spatial modeling, where one needs to know necessary admissibility conditions to propose new covariance models.
A series of works by Torquato and his coauthors in the field of materials science gathers known necessary conditions and illustrate them for many 2D and 3D theoretical models, along with reconstruction procedures, see~\cite{JiaStiTor07} and the survey~\cite[Sec.  2.2]{Tor02} and references therein.
This question was developed alongside in the field of geostatistics, where some authors do not tackle directly this issue, but address the realisability problem within some particular classes of models, e.g. Gaussian, mosaic, or Boolean model, see~\cite{Mas72, ChiDel99, Lantuejoul_geostatistical_simulation_2002, Eme10}.

A related question concerns the \textit{specific covariogram} of a stationary random set $X$, defined for all non empty bounded open sets $U\subset\R^d$ by
\begin{equation}
\gamma^s_{X}(y) 
= \frac{\E \leb(X\cap (y+X)\cap U)}{\leb(U)}
= \E \leb(X\cap (y+X)\cap (0,1)^{d}),
\label{eq:def_specfic_covariogram}
\end{equation}
where $\leb$ denotes the Lebesgue measure on $\R^d$.
The associated realisability problem, which consists in determining  whether there exists a stationary random set $X$ whose specific covariogram is a given function, is the \emph{(specific) covariogram realisability problem}.
Note that a straightforward Fubini argument gives that for any stationary random closed set $X$
\begin{equation}
\label{eq:equiv-2p-covariogram}
\gamma^{s}_{X}(y)=\int_{(0,1)^{d}}\P(x\in X,x-y\in X)dx=S_{2}(-y)=S_{2}(y),
\end{equation}
and thus the $S_2$ realisability problem and the specific covariogram problem are fundamentally the same.

Our main result provides an abstract and fully rigorous characterization of this problem for random measurable sets (RAMS) having locally finite mean perimeter.
Furthermore, in the restrictive one-dimensional case ($d=1$), results can be passed on to the classical framework of random closed sets.
It will become clear in this paper why the covariogram approach in the framework of random measurable sets is more adapted to a rigorous mathematical study.
Random measurable sets are an alternative to the classical random closed sets in stochastic geometry and geostatistics,
they provide a weaker framework allowing to manipulate more irregular functionals, such as the perimeter.
We therefore use the illustration provided by the $S_{2}$ problem to advocate the use of RAMS for solving theoretical problems of geometric nature.
Along the way, we extend the theory of random measurable sets, and in particular the local approximation of the perimeter by local covariograms.  
Remark that the framework of RAMS is related to the one of ``random sets of finite perimeter'' proposed recently by Rataj~\cite{Rataj_random_sets_of_finite_perimeter_2014}. 
However it is less restrictive since RAMS do not necessarily have finite perimeter.

Our main result uses a fundamental relation between the Lipschitz property of the covariogram function of a random set, and the finiteness of its mean variational perimeter, unveiled in \cite{Galerne_computation_perimeter_covariogram_2011}.
Like in~\cite[Th. 3.1]{LacMol} about point processes, 
we prove that the realisability of a given function $S_{2} : \mathbb{R}^{d}\rightarrow \R$ can be characterized by two independent conditions : a positivity condition, and a regularity condition, namely the Lipschitz property of $S_{2}$.
The positivity condition deals with the positivity of a linear operator extending $S_2$ on an appropriate space, and is of combinatorial nature.
The proof of this main result relies on a theorem dealing with positive operators on  a non-compact space recently derived in~\cite{LacMol} to treat realisability problems for point processes.
This general method therefore proves here its versatility by being applied in the framework of random sets in a very similar manner.

Checking whether $S_2$ satisfies the positivity condition is completely distinct from the concerns of this paper.
It is a difficult problem that has a long history.
It is more or less implicit in many  articles, and has been, to the best of the authors' knowledge, first  addressed directly by Shepp~\cite{She63}, later on by Matheron \cite{Mat93}, and more recently in~\cite{Qui08,Lac13b}. It is equivalent to the study of the \emph{correlation polytope} in the discrete geometry literature, see for instance the works of  Deza and Laurent~\cite{DezLau97}.
Still, a deep mathematical understanding of the problem remains out of reach.

The plan of the paper is as follows.
We give in the subsections below a quick overview of the mathematical objects involved here, namely random measurable sets, positivity, perimeter, and realisability problems, and we also state the main result of the paper dealing with the specific covariogram realisability problem for stationary random measurable sets with finite specific perimeter.
In Section~\ref{sec:rams}, we develop the theory of random measurable sets, define different notions of perimeter, and explore the relations with random closed sets, while Section~\ref{sec:perimeter_approximation} is devoted to the local covariogram functional and its use for perimeter approximation.
In Section~\ref{sec:proof}, we give the precise statement and the proof of the main result.
We also show that our main result extends to the framework of one-dimensional stationary RACS.

\subsection{Random measurable sets and variational perimeter}

Details about random measurable sets are presented in Section~\ref{sec:rams}, and we give here the essential notation for stating the results. 
Call $\sM$ the class of Lebesgue measurable sets of $\mathbb{R}^d$. 
A \textit{random measurable set} (RAMS) $X$ is a random variable taking values in $\sM$ endowed with the Borel $\sigma$-algebra induced by the local convergence in measure, which corresponds to the $L^1_\loc(\R^d)$-topology for the indicator functions, see Section~\ref{subsec:definition_rams} for details.
Remark that under this topology, one is bound to identify two sets $A$ and $B$ lying within the same Lebesgue class (that is, such that their symmetric difference $A\SymD B$ is Lebesgue-negligible), and we indeed perform this identification on $\sM$.
Say furthermore that a RAMS is \textit{stationary} if its law is invariant under translations of $\mathbb{R}^{d}$.
 
One geometric notion that can be extended to RAMS is that of perimeter.
For a deterministic measurable set $A$,
the perimeter of $A$ in an open set $U\subset\R^d$ is defined as the variation of the indicator function $\Ind_{A}$ in $U$, that is,
\begin{equation}
\label{eq:def-per-intro}
\Per(A;U) = 
\sup\left\{ \int_{U} \Ind_{A} (x) \operatorname{div}\varphi(x)dx:
~\varphi\in\mathcal{C}^1_c(U,\R^d),~\|\varphi(x)\|_2 \leq 1~\text{for all $x$} \right\},
\end{equation}
where $\mathcal{C}^1_c\left(U,\R^d\right)$
denotes the set of continuously differentiable functions $\varphi:U\rightarrow \R^d$ with compact support and $\|\cdot \|_{2}$ is the Euclidean norm~\cite{Ambrosio_Fusco_Pallara_functions_of_bounded_variation_free_discontinuity_problems_2000}, see Section~\ref{sec:rams-finite-perimeter} for a discussion and some properties of variational perimeters.
If $X$ is a RAMS, then, for all open sets $U\subset\R^d$, $\Per(X;U)$ is a well-defined random variable because the map $A\mapsto \Per(A;U)$ is lower semi-continuous for the local convergence in measure in $\R^d$~\cite[Proposition 3.38]{Ambrosio_Fusco_Pallara_functions_of_bounded_variation_free_discontinuity_problems_2000}.
Besides, if $X$ is stationary, then $U\mapsto \E(\Per(X;U))$ extends into a translation-invariant measure, and thus proportional to the Lebesgue measure.
One calls \textit{specific perimeter} or (\textit{specific variation}~\cite{Galerne2014_rfbv}) of $X$ the constant of proportionality that will be denoted by $\Per^{s}(X)$ and that is given by $\Per^{s}(X) = \E \Per(X;(0,1)^{d})$.
We refer to~\cite{Galerne2014_rfbv} for the computation of the specific perimeter of some classical random set models (Boolean models and Gaussian level sets).

\subsection{Covariogram realisability problems}

For a deterministic set $A$, one calls \textit{local covariogram} of $A$ the map
\begin{equation}
\label{eq:local-cov}
\begin{array}{ccl}
  \delta_{y;W}(A)  = \L^d(A\cap(y+A)\cap W), (y,W)\in \R^d\times \W,
\end{array}
\end{equation}
where $\W$ denotes the set of \textit{observation windows} defined by
$$
\W = \left\{ W\subset\R^d~\text{bounded open set such that}~\L^d(\partial W)=0 \right\}.
$$
Given a RAMS $X$, we denote by $\gamma_X(y;W) = \E \delta_{y;W}(X)$ the \textit{(mean) local covariogram} of $X$.
If $X$ is stationary, then the map $W\mapsto \gamma_X(y;W)$ is translation invariant and extends into a measure proportional to the Lebesgue measure.
Hence, one calls \textit{specific covariogram} of $X$ and denotes by $y\mapsto \gamma^s_X(y)$,
the map such that $\gamma_X(y;W) = \E \delta_{y;W}(X) = \gamma^s_X(y) \L^d(W)$.
Note that one simply has $\gamma^s_X(y) = \gamma_X(y,(0,1)^d)$.

We are interested in this paper in the specific covariogram realisability problem:
Given a function $S_2 : \mathbb{R}^{d} \rightarrow \R$, does there exists a stationary random measurable set $X\in \M $ such that
$S_2(y)  = \gamma^{s} _{X}(y)$ for all  $y\in \mathbb{R}^{d}$ ?

The specific covariogram candidate $S_2$ has to verify some structural necessary condition to be realisable.

%
%
 
\begin{definition}[Covariogram admissible functions]
\label{def:covariogram_admissible_measurable}
A function $\gamma : \R^d\times \W\rightarrow \R$ is said to be \emph{$\M $-local covariogram admissible}, or just \emph{admissible},
if for all 5-tuples $(q\geq 1, (a_i)\in \mathbb{R}^{q}  , (y_i)\in (\mathbb{R}^{d})^{q} , (W_i)\in \W^q, c\in  \mathbb{R})$,
$$
\left[\forall A\in \M,\quad c+\sum_{i=1}^{q}a_{i}\delta_{y_{i}; W_i}(A)\geq 0\right]
\Rightarrow
c+\sum_{i=1}^{q}a_{i}\gamma(y_i;W_i)\geq 0.
$$
A function $S_2:\R^d \rightarrow \R$ is said to be \emph{$\M $-specific covariogram admissible}, or just \emph{admissible}, if
the function $(y;W)\mapsto S_2(y) \L^d(W)$ is \emph{$\M $-local covariogram admissible}. 
\end{definition}

It is an immediate consequence of the positivity and linearity of the mathematical expectation that a realisable  $S_2$ function is necessarily admissible.
Checking whether a given $S_2$ is admissible, a problem of combinatorial nature, is difficult. It will not be addressed here, but as emphasized in equation \eqref{eq:equiv-2p-covariogram}, it is directly related  to the positivity problem for two-point covering functions, which is studied in numerous works, see \cite{DezLau97,She63,Mat93,Qui08,Lac13b}, and references therein.
Remark that being admissible is a strong constraint on $S_2$ that conveys the usual properties of covariogram functions, and in particular
$S_2(y)\geq 0$ for all $y\in\R^d$ (since for all $y\in\R^d$, $W\in\W$ and $A\in\M $, $\delta_{y;W}(A)\geq 0$).

In general, the admissibility of $S_{2}$ is not sufficient for $S_2$ to be realisable.
Consider the linear operator $\Phi$
\begin{equation}
\label{eq:Phi-def}
\Phi \left(c+\sum_{i=1}^{q}a_{i}\delta _{y_{i};W_{i}}\right)=c+\sum_{i=1}^{q}a_{i}S_2  (y_{i})\leb(W_{i}) 
\end{equation}
on the subspace of functionals on $\M $ generated by the constant functions and the covariogram evaluations $A\mapsto \delta_{y;W}(A)$, $y\in\R^d$, $W\in \W$.
The realisability of $S_{2}$ corresponds to the existence of a probability measure $\mu$ on $\M$ representing $\Phi$, i.e. such that 
$\Phi(g)=\int_{\M} g d\mu$ for $g$ in the aforementioned subspace.
In a non-compact space such as   $\M$, the positivity of $\Phi $, i.e. the admissibility of $S_{2}$, is not sufficient to represent it by a probability measure, as the $\sigma $-additivity  is also needed.

It has been shown in~\cite{LacMol} that in such non-compact frameworks, the realisability problem should better be accompanied with an additional regularity condition formulated in terms of a function called a \emph{regularity modulus}, see Section~\ref{sec:proof} for details. The perimeter function fulfills this role here, mostly because it can be approximated by linear combinations of covariograms, and has compact level sets.
The well-posed realisability problem with regularity condition we consider here deals with the existence of a stationary random measurable set $X\in \M$ such that 
$$
\begin{cases}
S_2(y) = \gamma^{s}_{X}(y),\quad y\in\R^d,\\
\Per^{s}(X) = \E \Per(X;(0,1)^{d}) < \infty.
\end{cases} 
$$

The main result of this paper is the following.
\begin{theorem}
\label{thm:intro-result}
Let $S_2:\R^d \mapsto \R$ be a function.
Then $S_2$ is the specific covariogram of a stationary random measurable set $X\in \M $ such that $\Per^{s}(X)<\infty$
if and only if $S_2$  is admissible and Lipschitz at $0$ along the $d$ canonical directions.
\end{theorem}

This result is analogous to the one obtained in~\cite{LacMol} for point processes, since the realisability condition is shown to be a positivity condition plus a regularity condition, namely the Lipschitz property of $S_2$.
As already discussed, a realisable function $S_2$ is necessarily admissible.
Besides, extending results from~\cite{Galerne_computation_perimeter_covariogram_2011}, 
we show that a stationary RAMS $X$ has a finite specific perimeter if and only if its specific covariogram $\gamma^s_X$ is Lipschitz, and we obtain an explicit relation between the Lipschitz 
constant of $S_2$ and the specific perimeter, see Proposition~\ref{prop:specific_covariogram_directional_lipschitz_constant}.
Hence the direct implication of Theorem~\ref{thm:intro-result} is somewhat straightforward.
The real difficulty consists in proving the converse implication. 
To do so we adapt the techniques of~\cite{LacMol} to our context which involves several technicalities regarding the approximation of the perimeter by linear combination of local covariogram functional.
We first establish the counterpart of Theorem~\ref{thm:intro-result} for the realisability of local covariogram function $\gamma :\R^d\times \W \rightarrow\R$ (see Theorem~\ref{thm:realisability_over_Rd_perBbeta}) and we then extend this result to the case of specific covariogram of stationary RAMS, see Theorem~\ref{thm:stationary-realisability_over_Rd_perB}.

In addition, we study the links between RAMS and the more usual framework of random closed sets (RACS), which \emph{in fine} enables us to obtain a result analogous to Theorem~\ref{thm:intro-result} for RACS of the real line (see Theorem~\ref{thm:realisability_for_1dRACS}), such a result was out of reach with previously developped methods.

\section{Random measurable sets}
\label{sec:rams}


\subsection{Definition of random measurable sets}
\label{subsec:definition_rams}

Random measurable sets (RAMS) are defined as random variables taking value in the set $\M$ of Lebesgue (classes of) sets of $\R^d$ endowed with the Borel $\sigma$-algebra $\B(\M)$ induced by the natural topology, the so-called local convergence in measure.
We recall that a sequence of measurable sets $(A_n)_{n\in\N}$ locally converges in measure to a measurable set $A$ if for all bounded open sets $U\subset \R^d$, the sequence $\L^d\left( (A_n\Delta A)\cap U\right)$ tends to $0$, where $\Delta$ denotes the symmetric difference.
The local convergence in measure simply corresponds to the convergence of the indicator functions $\Ind_{A_n}$ towards $\Ind_A$ in the space of locally integrable functions $L^1_\loc(\R^d)$, and consequently $\M$ is a complete metrizable space$^1$\footnote{$^1$
This is a consequence of the facts that $L^1_\loc(\R^d)$ is a complete metrizable space and that the set of indicator functions is closed in $L^1_\loc(\R^d)$.}.

\begin{definition}[Random measurable sets]
A \emph{random measurable set} (RAMS) $X$ is a measurable map
$X:\omega \mapsto X(\omega)$ from $(\Omega,\mathcal{A})$ to $\left(\M,\B(\M)\right)$,
where $\B(\M)$ denotes the Borel $\sigma$-algebra induced by the local convergence in measure.
\end{definition}

Note that if $X$ is a RAMS, then $\omega\mapsto \Ind_{X(\omega)}$ is a random locally integrable function.
This concept of random measurable (class of) set(s) is not standard, 
and, to the best of the authors knowledge, it was first introduced in~\cite{Straka_Stepan_random_sets_in_01_1987}
for random subsets of the real interval $[0,1]$, as mentioned in~\cite{Molchanov_theory_random_sets_2005}.

In the remaining part of this section, we will discuss the link between RAMS and other classical random objects, namely random Radon measures, measurable subsets of $\Omega\times \R^d$, and random closed sets.

%

\paragraph{Random Radon measures associated with random measurable sets}
Following the usual construction of random objects,
a random Radon measure is defined as a measurable function from a probability space $(\Omega,\mathcal{A},\P)$
to the space $\mathbf{M}^{+}$ of positive Radon measures on $\R^d$ equipped with the smallest $\sigma$-algebra for which the evaluation maps
$
\mu \mapsto \mu(B),~B\in\B(\R^d)~\text{relatively compact},
$
are measurable, see \textit{e.g}~\cite{Daley_Vere_Jones_introduction_theory_point_process_II_2008, Kallenberg_random_measures_1986, Schneider_Weil_stochastic_and_integral_geometry_2008}.
Any RAMS $X\subset\R^d$ canonically defines a random Radon measure that is the restriction to $X$ of the Lebesgue measure, that is, 
$B \mapsto \leb(X\cap B)$ for Borel set $B\in \B(\R^d)$.
The measurability of this restriction results from the observation that, for all $B\in \B(\R^d)$, the map $f\mapsto \int_B f(x) dx$ is measurable for the $L^1_\loc$-topology.

\paragraph{Existence of a measurable graph representative} 
For a RAMS $X:\Omega\rightarrow \M$,
one can study the measurability properties of the graph $Y=\{(\omega ,x): x\in X(\omega)\}\subset \Omega \times \mathbb{R}^{d}$.

\begin{definition}[Measurable graph representatives]
A subset $Y\subset \Omega\times \R^d$ is a \textit{measurable graph representative} of a RAMS $X$ if
\begin{enumerate}
\item $Y$ is a measurable subset of $\Omega\times \R^d$
(\emph{i.e.} $Y$ belongs to the product $\sigma$-algebra $\mathcal{A}\otimes\mathcal{B}(\R^d)$),
\item For a.a. $\omega\in\Omega$, the $\omega$-section $Y(\omega) = \{x\in\R^d:~(\omega,x)\in Y\}$ is equivalent in measure to $X(\omega)$, \emph{i.e.} $\leb(Y(\omega)\Delta X(\omega))=0$.
\end{enumerate}
\end{definition}

\begin{proposition}
\label{prop_random_measurable_sets_jointly_measurable_sets}
Any measurable set $Y\in\mathcal{A}\otimes\mathcal{B}(\R^d)$ canonically defines a RAMS by considering the Lebesgue class of its $\omega$-sections:
$$
\omega \mapsto Y(\omega) = \{x\in\R^d:~(\omega,x)\in Y\}.
$$
Conversely, any RAMS $X$ admits measurable graph representatives
$Y \in \mathcal{A}\otimes\mathcal{B}(\R^d)$.
\end{proposition}

\begin{proof}
The first point is trivial. Let us prove the second point.
Consider the random Radon measure $\mu$ associated to $X$, that is
$$
\mu(\omega,B) = \leb(X(\omega)\cap B) = \int_B \Ind_{X(\omega)}(x) dx.
$$
By construction this random Radon measure is absolutely continuous with respect to the Lebesgue measure.
Then, according to the Radon-Nikodym theorem for random measures (see Theorem~\ref{thm_radon_nikodym_random_measure_ac_lebesgue} in Appendix),
there exists a jointly measurable map
$g :(\Omega\times\R^d, \mathcal{A} \otimes \B(\R^d)) \rightarrow \R$
such that for all $\omega\in\Omega$,
$$
\mu(\omega,B) = \int_B g(\omega,x) dx, \quad B\in\B(\R^d).
$$
Hence for all $\omega\in\Omega$, $\Ind_{X(\omega)}(\cdot)$ and $g(\omega,\cdot)$ are both Radon-Nikodym derivatives of $\mu(\omega,\cdot)$ and thus are equal almost everywhere.
In particular, for a.a. $x\in\R^d$, $g(\omega,x)\in\{0,1\}$.
Consequently, the function $(\omega,x) \mapsto \Ind(g(\omega,x)=1)$ is also jointly measurable and is a Radon-Nikodym derivative of $\mu(\omega,\cdot)$ for all $\omega\in\Omega$, and thus
the set
$$
Y = \{ (\omega,x) \in \Omega\times\R^d:~g(\omega,x)=1\}
$$
is a measurable graph representative of $X$.
\end{proof}


\paragraph{Random measurable sets and random closed sets}
Recall that $(\Omega, \mathcal{A},\P)$ denotes our probability space.
Let $\F = \F\left( \R^d \right)$ be the set of all closed subsets of $\R^d$.
Following~\cite[Definition 1.1]{Molchanov_theory_random_sets_2005} a random closed set is defined as follows.

\begin{definition}[Random closed sets]
A map $Z:\Omega \rightarrow \F$ is called a \emph{random closed set} (RACS) if for every compact set $K\subset \R^d$,
$\{ \omega:~Z(\omega)\cap K \neq \emptyset \} \in \mathcal{A}$.
\end{definition}

The framework of random closed sets is standard in stochastic geometry~\cite{Matheron_random_sets_and_integral_geometry_1975, Molchanov_theory_random_sets_2005}.
Let us  reproduce a result of C.J. Himmelberg that allows to link the different notions of random sets, see~\cite[Theorem 2.3]{Molchanov_theory_random_sets_2005} or the original paper~\cite{Himmelberg_measurable_relations_1975} for the complete theorem.
\begin{theorem}[Himmelberg]
\label{thm_himmelberg}
Let $(\Omega, \mathcal{A},\P)$ be a probability space and $Z:\Omega \rightarrow Z(\omega)\in \F$ 
be a map taking values into the set of closed subsets of $\R^d$.
Consider the two following assertions:
\begin{enumerate}
\item[(i)] $\{\omega:~Z\cap F \neq \emptyset\} \in \mathcal{A}$ for every closed set $F\subset\R^d$,
\item[(ii)] The graph of $Z$, \textit{i.e.} the set $\{ (\omega,x)\in\Omega\times \R^d:~x\in Z(\omega)\}$,
belongs to the product $\sigma$-algebra $\mathcal{A}\otimes \mathcal{B}(\R^d)$,
\end{enumerate}
Then the implication $(i)\Rightarrow(ii)$ is always true, and if the probability space $(\Omega, \mathcal{A},\P)$ is complete, one has the equivalence $(i)\Leftrightarrow(ii)$.
\end{theorem}

In view of our definitions for random sets, Himmelberg's theorem can be rephrased in the following terms.

\begin{proposition}[RACS and closed RAMS]
\label{prop_a_closed_random_measurable_set_is_a_racs}
\begin{enumerate}
\item[(i)] Any RACS $Z$ has a measurable graph $Y = \{ (\omega,x)\in\Omega\times \R^d:~x\in Z(\omega)\}$, 
and thus also defines a unique random measurable set.

\item[(ii)]
Suppose that the probability space $(\Omega, \mathcal{A},\P)$ is complete.
Let $Y\in \mathcal{A}\otimes\mathcal{B}(\R^d)$ be a measurable set such that for all $\omega\in\Omega$,
its $\omega$-section
$
Y(\omega) = \{x\in\R^d:~(\omega,x)\in Y\}
$ 
is a closed subset of $\R^d$.
Then, the map $\omega\mapsto Y(\omega)$ defines a random closed set.
\end{enumerate}
\end{proposition}

\subsection{Random measurable sets of finite perimeter}
\label{sec:rams-finite-perimeter}

%

For a closed set $F$, the perimeter is generally defined by the $(d-1)$-dimensional measure of the topological boundary, that is $\mathcal{H}^{d-1}(\partial F)$.
This definition is not relevant for a measurable set $A\subset\R^d$, 
in the sense that the value $\mathcal{H}^{d-1}(\partial A)$ strongly depends on the representative of $A$ within its Lebesgue class.
The proper notion of perimeter for measurable sets is the variational perimeter that defines the perimeter as the variation of the indicator function of the set.
An important feature of the variational perimeter is that it is lower semi-continuous for the convergence in measure, while the functional
$F \mapsto \mathcal{H}^{d-1}(\partial F)$ is not lower semi-continuous on the set of closed sets $\F$ endowed with the hit or miss topology.
This is a key aspect for this paper since it allows to consider the variational perimeter as a regularity modulus for realisability problems in following the framework of~\cite{LacMol}. 

\paragraph{Variational perimeters}
Let $U$ be an open subset of $\R^d$.
Recall that the \textit{(variational) perimeter} $\Per(A;U)$ of a measurable set $A\in\M$ in the open set $U$ is defined by~\eqref{eq:def-per-intro}.
Denote by $S^{d-1}$ the unit sphere of $\R^d$.
Closely related to the perimeter, one also defines the \textit{directional variation} in the direction $u\in S^{d-1}$ of $A$ in $U$ by~\cite[Section 3.11]{Ambrosio_Fusco_Pallara_functions_of_bounded_variation_free_discontinuity_problems_2000}
$$
V_u(A;U) =
\sup\left\{ \int_{U} \Ind_{A} (x) \langle \nabla \varphi (x), u \rangle dx
:~\varphi\in\mathcal{C}^1_c(U,\R),~|\varphi(x)| \leq 1 ~ \text{for all $x$} \right\}.
$$
For technical reasons, we also consider the \textit{anisotropic perimeter}
$$
A \mapsto \Per_{\mathbf{B}}(A;U) = \sum_{j=1}^d V_{e_j}(A;U)
$$
which adds up the directional variations along the $d$ directions of the canonical basis $\mathbf{B} = \{e_1,\dots,e_d\}$.
In geometric measure theory, the functional $A\mapsto \Per_{\mathbf{B}}(A;U)$ is described as the anisotropic perimeter associated with the anisotropy function $x\mapsto \|x\|_\infty$, see e.g.~\cite{Caselles_et_al_convex_calibrable_sets_anisotropic_norms_2008} and the references therein.
Indeed, one easily sees that
$$
\Per_{\mathbf{B}}(A;U) = \sup\left\{ \int_{U} \Ind_{A} (x) \operatorname{div}\varphi(x)dx:~\varphi\in\mathcal{C}^1_c(U,\R ^d),~\|\varphi(x)\|_\infty \leq 1~\text{for all $x$} \right\}.
$$
Hence, the only difference between the variational definition of the isotropic perimeter $\Per(A;U)$ and the one of the anisotropic perimeter $\Per_{\mathbf{B}}(A;U)$ is that the test functions $\varphi$ take values in the $\ell_2$-unit ball $B_d$ for the former whereas they take values in the $\ell_\infty$-unit ball $[-1,1]^d$ for the latter.
The set inclusions $B_d \subset [ -1,1]^d \subset \sqrt{d} B_d$ lead to the tight inequalities
\begin{equation}
\label{eq:relation-Per-PerB}
\Per(A;U) \leq \Per_{\mathbf{B}}(A;U) \leq \sqrt{d} \Per(A;U).
\end{equation}
Consequently a set $A$ has a finite perimeter $\Per(A;U)$ in $U$ if and only if it has a finite anisotropic perimeter
$\Per_{\mathbf{B}}(A;U)$, let us mention that this equivalence is not true when considering only one directional variation $V_u(A;U)$.
One says that a measurable set $A\subset\R^d$ has \textit{locally finite perimeter} if $A$ has a finite perimeter $\Per(A;U)$ in all bounded open sets $U\subset\R^d$.


To finish, let us mention that if $X$ is a RAMS then $\Per(X;U)$, $\Per_{\mathbf{B}}(X;U)$, and $V_u(X;U)$, $u\in S^{d-1}$, are well-defined random variables since the maps $A\mapsto \Per(A;U)$, $A\mapsto \Per_{\mathbf{B}}(A;U)$ and $A\mapsto V_u(A;U)$ are lower semi-continuous for the convergence in measure~\cite{Ambrosio_Fusco_Pallara_functions_of_bounded_variation_free_discontinuity_problems_2000}.
Consequently one says that a RAMS $X$ has a.s. finite (resp. locally finite) perimeter in $U$ if the random variable $\Per(X;U)$ is a.s. finite 
(resp. if, for all bounded open sets $V\subset U$, $\Per(X;V)$ is a.s. finite).

\begin{remark}
Rataj recently proposed a framework for ``random sets of finite perimeter''~\cite{Rataj_random_sets_of_finite_perimeter_2014} that models random sets as random variables in the space of indicator functions of sets of finite perimeter endowed with the Borel $\sigma$-algebra induced by the strict convergence in the space of functions of bounded variation~\cite[Section 3.1]{Ambrosio_Fusco_Pallara_functions_of_bounded_variation_free_discontinuity_problems_2000}.
Since this convergence induced the $L^1$-convergence of indicator functions, any ``random set of finite perimeter'' $X$ uniquely defines a RAMS $X$ having a.s. finite perimeter.
One advantage of the RAMS framework is that it is more general in the sense that it enables to consider random sets that do not have finite perimeter, see e.g. Corollary~\ref{cor:tcd-per}.
\end{remark}

\paragraph{Closed representative of one-dimensional sets of finite perimeter}
Although the general geometric structure of sets of finite perimeter is well-known
(see~\cite[Section 3.5]{Ambrosio_Fusco_Pallara_functions_of_bounded_variation_free_discontinuity_problems_2000}),
it necessitates involved notions from geometric measure theory (rectifiable sets, reduced and essential boundaries, etc.).
However, when restricting to the case of one-dimensional sets of finite perimeter, all the complexity vanishes since subsets of $\R$ having finite perimeter all correspond to finite unions of non empty and disjoint closed intervals.



More precisely, according to Proposition 3.52 of~\cite{Ambrosio_Fusco_Pallara_functions_of_bounded_variation_free_discontinuity_problems_2000},
if a non-negligible measurable set $A\subset \R$ has finite perimeter in an interval $(a,b)\subset\overline{\R}$, there exists an integer $p$ and $p$ pairwise disjoint non empty and closed intervals $J_i = [a_{2i-1}, a_{2i}]\subset \overline{\R}$, with $a_1 < a_2 <\dots < a_{2p}$,
such that
\begin{itemize}
\item $A\cap (a,b)$ is equivalent in measure to the union $\bigcup_i J_i$,
\item the perimeter of $A$ in $(a,b)$ is the number of interval endpoints belonging to $(a,b)$,
$$
\Per(A; (a,b)) = \# \{a_1,a_2,\dots, a_{2p}\}\cap (a,b).
$$
\end{itemize}
Remark that a set of the form $A = \bigcup_i [a_{2i-1}, a_{2i}]$ is closed, and that such a set satisfies the identity
$\Per(A; (a,b)) = \mathcal{H}^0(\partial A\cap (a,b))$,
where $\partial A$ denotes the topological boundary of $A$ and $\mathcal{H}^0$ is the Hausdorff measure of dimension $0$ on $\R$ (\emph{i.e.} the counting measure), while in the general case one only has $\Per(A; (a,b)) \leq \mathcal{H}^0(\partial A\cap (a,b))$ since $A$ may contain isolated points.

In the general case, since  $A\subset\R$ may have locally finite perimeter, then there exists a unique countable or finite family of closed and disjoint intervals $J_i = [a_{2i-1}, a_{2i}]$, $i\in I\subset \Z$, such that $A$ is equivalent in measure to $\bigcup_{i\in I} J_i$, and for all bounded open intervals $(a,b)$, $\Per(A;(a,b))$ is the number of interval endpoints belonging to $(a,b)$.

Using both this observations and Proposition~\ref{prop_a_closed_random_measurable_set_is_a_racs}, one obtains the following proposition.

\begin{proposition}
\label{prop_1d_racs_representative}
Suppose that the probability space $(\Omega, \mathcal{A},\P)$ is complete.
Let $X$ be a RAMS of $\R$ that has a.s. locally finite perimeter.
Then, there exists a RACS $Z\subset\R$ such that for $\P$-almost all $\omega\in\Omega$ and for all $a<b \in\R$,
$$
\mathcal{L}^1(X(\omega)\SymD Z(\omega)) = 0
\quad \text{and} \quad
\Per(X(\omega);(a,b)) = \mathcal{H}^0(\partial Z(\omega)\cap (a,b)).
$$
\end{proposition}
 

\begin{proof}
First remark that   a measurable set of finite perimeter 
 $A\subset\R$ equivalent in measure to $\bigcup_{i\in I} [a_{2i-1}, a_{2i}]$ for some finite or countable index set $I\subset \Z$ has the Lebesgue density
\begin{equation*}
\begin{aligned}
D(x,A) & = \lim_{r\rightarrow 0+} \frac{\L^1(A\cap (x-r,x+r))}{2r}
\\
& =
\begin{cases}
1 & \text{if $x$ is in some open interval $(a_{2i-1}, a_{2i})$},
\\
\frac{1}{2} & \text{if $x$ is an interval endpoint $a_{2i-1}$ or $a_{2i}$ for some $i\in I$},
\\
0 & \text{if $\displaystyle x\notin\bigcup_{i\in I} [a_{2i-1}, a_{2i}]$}.
\end{cases}
\end{aligned}
\end{equation*}

Let $X$ be a RAMS of $\R$ that has a.s. locally finite perimeter.
Let $\Omega'\in\mathcal{A}$ be a subset of $\Omega$ of probability one such that for all
$\omega\in\Omega'$, $X$ has locally finite perimeter.
For all $\omega\in\Omega'$, the Lebesgue class $X(\omega)$ admits a representative that is the union of 
an at most countable family of non empty and disjoint closed intervals.
According to the above observation, for a fixed $\omega\in\Omega'$,
the density $D(x,X(\omega))$ exists for all $x\in\R$, and the good representative of $X$ is given by
$\{ x\in \R:~ D(x,X(\omega)) > 0 \}$.
Let $(r_n)_{n\in\N}$ be a positive sequence decreasing to $0$, and  define for all $\omega\in\Omega$,
$$
g(\omega,x) = \limsup_{n\rightarrow+\infty} \frac{\L^1(X(\omega)\cap (x-r_n,x+r_n))}{2r_n}.
$$
According to the proof of Theorem~\ref{thm_radon_nikodym_random_measure_ac_lebesgue} and Proposition~\ref{prop_random_measurable_sets_jointly_measurable_sets},
$$
Y = \{ (\omega,x) \in \Omega\times\R^d:~g(\omega,x)>0\}
$$
is a measurable set that is a measurable graph representative of $X$.
Besides, for a given $\omega\in\Omega'$, since $D(x,X(\omega))$ exists for all $x\in\R$,
$$
D(x,X(\omega)) = g(\omega,x),\quad x\in\R.
$$
Hence, for all $\omega\in\Omega'$, the $\omega$-section $Y(\omega)$ of $Y$ is the union of 
an at most countable and locally finite family of non empty and disjoint closed intervals, and in particular a closed set.
Thus, by Proposition~\ref{prop_a_closed_random_measurable_set_is_a_racs},
the map $\omega\mapsto Y(\omega)$ defines a random closed set.
\end{proof}

\paragraph{Non-closed RAMS in dimension $d>1$}
In contrast to the one-dimensional case,
in dimension $d>1$ there exist measurable sets of finite perimeter that do not have closed representative in their Lebesgue class.
A set $A$ obtained as the union of an infinite family of open balls with small radii and with centers forming a dense subset of $[0,1]^{d}$ is considered in~\cite[Example 3.53]{Ambrosio_Fusco_Pallara_functions_of_bounded_variation_free_discontinuity_problems_2000}. It has finite perimeter, finite measure $\mathcal{L}^{d}(A)<1$, and is such that $\mathcal{L}^{d}(A\cap U)>0$ for any open subset $U$ of $[0,1]^{d}$.

Such a set clearly has no closed representative, because if it had one, say $F$, then $F$ would charge every open subset of $[0,1]^{d}$, and therefore it would be dense in $[0,1]^{d}$.
Since $F$ is closed, one would have $F = [0,1]^{d}$, which contradicts $\mathcal{L}^{d}(F) = \mathcal{L}^{d}(A)<1$.  


\section{Local covariogram and perimeter approximation}
\label{sec:perimeter_approximation}

In this section, we establish general properties of the local covariogram of a measurable set, as well as the mean local covariogram of a RAMS.
A particular emphasis is given on the relation between the local perimeter of a set and the Lipschitz constant of its local covariogram in order to adapt the results of~\cite{Galerne_computation_perimeter_covariogram_2011} to the local covariogram functional.

\subsection{Definition and continuity}

The local covariogram of a measurable set $A\in \M$ is defined in~\eqref{eq:local-cov}.
Remark that for all $y\in\R^d$, and $W\in\W$,
\begin{equation}
\delta_{y;W}(A) = \delta _{y;W}(A\cap(W\cup (-y+W))),\quad A\in\M,
\label{eq:locality_local_covariogram}
\end{equation}
so that only the part of $A$ included in the domain $W\cup (-y+W)$ has an influence on the value of $\delta_{y;W}(A)$, hence local covariograms are indeed local.
Before enunciating specific results of interest for our realisability problem, let us prove that local covariograms are continuous for the local convergence in measure.

\begin{proposition}[Continuity of local covariograms]\quad
\label{prop:continuity_local_covariogram}
\begin{enumerate}
\item[(i)] For all $A\in \M$ and $W\in\W$, the map $y\mapsto \delta_{y;W}(A)$ is uniformly continuous over $\R^d$.
 
\item[(ii)] Let $A\in \M$ and $y\in\R^d$. Then, for all $U, W \in\W$,
$$
|\delta_{y;U}(A) - \delta_{y;W}(A)| \leq \L^d(U\Delta W).
$$
In particular, the map $W\mapsto \delta_{y;W}(A)$ is continuous for the convergence in measure.

\item[(iii)] Let $A$, $B\in \M$ and let $W\in\W$.
Then, for all $y\in\R^d$,
$$
\left| \delta_{y;W}(A) - \delta_{y;W}(B) \right| \leq 2 \L^d((A\Delta B)\cap (W\cup(-y+W))).
$$
In particular, the map $A\mapsto \delta_{y;W}(A)$ is continuous for the local convergence in measure.
\end{enumerate}
\end{proposition}

\begin{proof}
\textsl{(i)}
The convolution interpretation for local covariograms yields
$$
\delta_{y;W}(A) = \int_{\R^d} \Ind_{A\cap W}(x) \Ind_{-A}(y-x) dx = \Ind_{A\cap W}\ast\Ind_{-A}(y).
$$
Since $\Ind_{A\cap W}\in L^1(\R^d)$ and $\Ind_{-A} \in L^\infty(\R^d)$, the uniform continuity is ensured by the $L^p$-$L^{p'}$-convolution theorem, see e.g.~\cite[Proposition 3.2]{Hirsch_Lacombe_elements_functional_analysis_1999}.

\textsl{(ii)}
Using the general inequality $|\L^d(A_1) - \L^d(A_2)| \leq \L^d(A_1\Delta A_2)$, one gets
$$
|\delta_{y;U}(A) - \delta_{y;W}(A)| \leq \L^d( (A\cap (A+y)\cap U) \Delta (A\cap (A+y)\cap W) )
\leq \L^d(U\Delta W).
$$

\textsl{(iii)}
If $A$ and $B$ have finite Lebesgue measure,  
$$
\begin{aligned}
\left| \delta_{y;W}(A) - \delta_{y;W}(B) \right|
& = \left| \Ind_{A\cap W}\ast\Ind_{-A}(y) - \Ind_{B\cap W}\ast\Ind_{-B}(y) \right|
\\
& \leq \left| \Ind_{A\cap W}\ast\Ind_{-A}(y) - \Ind_{A\cap W}\ast\Ind_{-B}(y) + \Ind_{A\cap W}\ast\Ind_{-B}(y) - \Ind_{B\cap W}\ast\Ind_{-B}(y) \right|
\\
& \leq \left| \Ind_{A\cap W}\ast(\Ind_{-A} - \Ind_{-B})(y)\right| + \left|(\Ind_{A\cap W}- \Ind_{B\cap W})\ast\Ind_{-B}(y) \right|
\\
& \leq \|\Ind_{A\cap W}\|_\infty \|\Ind_{-A} - \Ind_{-B}\|_1 + \|\Ind_{A\cap W}- \Ind_{B\cap W}\|_1 \|\Ind_{-B}\|_\infty
\\
& \leq \L^d(A\Delta B) + \L^d((A\cap W)\Delta (B\cap W))
\\
& \leq 2 \L^d(A\Delta B).
\end{aligned}
$$
The announced general inequality is obtained from~\eqref{eq:locality_local_covariogram} which ensures that one can replace $A$ and $B$ by $A\cap (W\cup(-y+W))$ and $B\cap (W\cup(-y+W))$ without changing the values of 
$\delta_{y;W}(A)$ and $\delta_{y;W}(B)$.
\end{proof}

\subsection{Local covariogram and anisotropic perimeter}

As for the case of covariogram~\cite{Galerne_computation_perimeter_covariogram_2011}, difference quotients at zero of local covariograms are related to the directional variations of the set $A$.
This is clarified by the next results, where $V_u(f;U)$ denotes the directional variation of $f\in L^1(U)$ in $U$ in the direction $u\in S^{d-1}$, that is
$$
V_u(f;U) =
\sup\left\{ \int_{U} f (x) \langle \nabla \varphi (x), u \rangle dx
:~\varphi\in\mathcal{C}^1_c(U,\R),~|\varphi(x)| \leq 1 ~ \text{for all $x$} \right\}.
$$
Recall that by definition, for a set $A\in\M$, $V_u(A;U)$ stands for $V_u(\Ind_A;U)$),
and that
$
A\ominus B
$ denotes the Minkowski difference of two measurable sets $A$ and $B$.
The following proposition states a well-known result from the theory of functions of bounded variation that is fully proved in~\cite{Galerne2014_rfbv}.

\begin{proposition}
\label{prop:integral_diff_quotient_and_directional_variation}
Let $U$ be an open subset of $\R^d$ and $u\in S^{d-1}$.
Then, for all functions $f\in L^1(U)$ and $\eps\in\R$,
$$
\int_{U\ominus[0, \eps u]} \frac{\left| f(x+\eps u) - f(x) \right|}{|\eps|} dx \leq V_u(f;U),
$$
where $[0, \eps u]$ denotes the segment $\{t \eps u:~t\in[0,1]\}$,
and
\begin{equation}
\lim_{\eps\rightarrow 0} \int_{U\ominus[0, \eps u]} \frac{\left| f(x+\eps u) - f(x) \right|}{|\eps|} dx = V_u(f;U).
\label{eq_directional_variation_equals_limit_int_abs_diff_quotient}
\end{equation}
\end{proposition}

The next two propositions show that when $f$ is the indicator function of a set $A$, the integral
$$
\int_{U\ominus[0, \eps u]} \frac{\left| f(x+\eps u) - f(x) \right|}{|\eps|}dx
$$
can be expressed as a linear combination of local covariograms $\delta_{y;W}(A)$.
Since this linear combination will be central in the next results, we introduce the notation
\begin{equation*}
\sigma _{u;W}(A)=\frac{1}{\|u\|}\left(\delta_{0;W\ominus[-  u,0]}(A) -  \delta_{  u;W\ominus[-  u,0]}(A) + \delta_{0;W\ominus[0,  u]}(A) - \delta_{-  u;W\ominus[0,  u]}(A)\right)
\end{equation*}
for any $A\in\M$, $u\neq 0$, and $W\in \W$.
Remark that for $W\in \W,y\in \mathbb{R}^{d},A\in  \M(\mathbb{R}^{d})$,
\begin{equation}
\label{eq:difference_zero_local_covariogram_is_leb_measure}
\delta_{0;W}(A) - \delta_{y;W}(A)
= \L^d(A\cap W) - \L^d(A\cap(y+A)\cap W) = \L^d((A\setminus(y+A))\cap W).
\end{equation}

\begin{proposition}[Local covariogram and anisotropic perimeter]
\label{prop:per-approx}
For all $A\in\M$, $W\in\W$, $\eps\in\R$, and $u\in S^{d-1}$,
\begin{equation}
0\leq  \sigma_{\eps u;W}(A)\leq V_{u}(A;W)
\quad \quad  
\text{and}
\quad \quad 
\lim_{\eps\rightarrow 0}
\sigma_{\eps u;W}(A)= V_{u}(A;W).
\label{eq:approx-perimeter}
\end{equation}
When summing along the $d$ directions of the canonical basis $\mathbf{B} = \{e_1,e_2,\dots,e_d\}$, one obtains similar results for the anisotropic perimeter, that is, for all $A\in\M$ and $\eps\in\R$,
$$
0\leq \sum_{j=1}^d \sigma _{\eps e_{j};W}(A)\leq \PerB(A;W)
\quad \quad 
\text{and}
\quad \quad 
\lim_{\eps\rightarrow 0}
\sum_{j=1}^d
 \sigma_{\eps e_{j};W}(A)
=  \PerB(A;W).
$$
\end{proposition}

\begin{proof}
The announced inequalities are immediate from Proposition~\ref{prop:integral_diff_quotient_and_directional_variation} and the equality
$$
\begin{aligned}
& \int_{W\ominus[0,\eps u]} |\Ind_A(x+\eps u) - \Ind_A(x)| dx
= |\eps| \sigma _{\eps u;W}(A)
\end{aligned}
$$ holds for all $A\in\M, W\in\W, u\in S^{d-1}, \eps\in\R$. Indeed,
$$
\begin{aligned}
& \int_{W\ominus[0,\eps u]} |\Ind_A(x+\eps u) - \Ind_A(x)| dx\\
& = \int_{W\ominus[0,\eps u]} |\Ind_{-\eps u + A}(x) - \Ind_A(x)| dx \\
& = \L^d\left( \left( (-\eps u + A)\Delta A \right)\cap (W\ominus[0,\eps u])\right)\\
& = \L^d\left( \left( (-\eps u + A)\setminus A \right)\cap (W\ominus[0,\eps u])\right) + \L^d\left( \left( A\setminus (-\eps u + A) \right)\cap (W\ominus[0,\eps u])\right).\\
\end{aligned}
$$
Applying the translation by vector $\eps u$ yields
$$
\L^d\left( \left( (-\eps u + A)\setminus A \right)\cap (W\ominus[0,\eps u])\right)
= \L^d\left( \left( A \setminus (\eps u + A) \right)\cap \left(\eps u +(W\ominus[0,\eps u])\right)\right).
$$
Remark that $\eps u + (W\ominus[0,\eps u]) = W\ominus[-\eps u,0]$ and thus, using (\ref{eq:difference_zero_local_covariogram_is_leb_measure}),
$$
\begin{aligned}
& \int_{W\ominus[0,\eps u]} |\Ind_A(x+\eps u) - \Ind_A(x)| dx\\
& = 
\L^d\left( \left( A \setminus (\eps u + A) \right)\cap (W\ominus[-\eps u,0])\right)
+ \L^d\left( \left( A\setminus (-\eps u + A) \right)\cap (W\ominus[0,\eps u])\right)
\\
& = \delta_{0;W\ominus[-\eps u,0]}(A) -  \delta_{\eps u;W\ominus[-\eps u,0]}(A) + \delta_{0;W\ominus[0,\eps u]}(A) - \delta_{-\eps u;W\ominus[0,\eps u]}(A).
\end{aligned}
$$
\end{proof}

We  turn to the counterpart of Proposition~\ref{prop:per-approx} for mean local covariograms of RAMS.
For a RAMS $X$, $\gamma_X$ denotes the \emph{(mean) local covariogram} of the RAMS $X$ defined by 
$\gamma_X (y;W)=\E \delta_{y;W}(X)$, $y\in \mathbb{R}^{d},$ $W\in \W$,
and define similarly ${\sigma_{X}(u;W)=\E \sigma_{u;W}(X)}$.

\begin{corollary}
\label{cor:tcd-per}
Let $X$ be a RAMS.
Then, for all $W\in\W$ and $u\in S^{d-1}$,
$$
\sigma _{X}(\eps u;W) \leq \E V_{u}(X;W)
\quad \quad  
\text{and}
\quad \quad 
\E V_{u}(X;W) =\lim_{\eps \to 0} \sigma _{X}(\eps u;W).
$$
\end{corollary}

\begin{proof}
This is straightforward from~\eqref{eq:approx-perimeter}.
If $\E V_{u}(X;W)<\infty$, apply the Lebesgue theorem with the almost sure convergence and domination given by~\eqref{eq:approx-perimeter},
while if $\E V_{u}(X;W)=\infty$, apply Fatou's lemma.
\end{proof}

We  turn to similar results for stationary RAMS.
First recall that if $X$ is a stationary RAMS,
the specific covariogram of $X$ is defined by $\gamma_{X}^{s}(y) = \gamma_{X}(y;(0,1)^d)\in[0,1]$.
By analogy with the specific perimeter $\Per^{s}(X) = \E \Per(X;(0,1)^{d})$,
the \textit{specific anisotropic perimeter} of $X$ is defined by $\PerB^{s}(X) = \E \PerB(X;(0,1)^{d})\in[0,+\infty]$ and 
for all $u\in S^{d-1}$, the \textit{specific variation} of $X$ in direction $u$ is given by $ V^s_{u}(X) = \E V_{u}(X;(0,1)^d)$.

%

For a function $F:\R^d\rightarrow \R$, define the Lipschitz constant in the $j$-th direction at $y\in \mathbb{R}^{d}$ by
$$
\Lip_{j}( F,y) = \sup_{ t\in\R} \frac{|F(y+te_j)-F(y)|}{|t|},
$$
and denote $\Lip_{j}(F)=\sup_{y\in \mathbb{R}^{d}}\Lip_{j}(F,y)$.
Note that a function $F$ is Lipschitz in the usual sense if and only each constant $\Lip_{j}(F)$ is finite for $j= 1,\dots,d$.

\begin{proposition}
\label{prop:specific_covariogram_directional_lipschitz_constant}
Let $X$ be a stationary RAMS and let $\gamma_{X}^{s}$ be its specific covariogram.
Then $\gamma_{X}^{s}$ is even and, for all $y,z\in\R^d$,
$$
|\gamma_{X}^{s}(y)-\gamma_{X}^{s}(z)| \leq \gamma_{X}^{s}(0) - \gamma_{X}^{s}(y-z).
$$
In particular, $\gamma_{X}^{s}$ is Lipschitz over $\R^d$ if and only if $\gamma_{X}^{s}$ is Lipschitz at $0$.
Besides,  for all $j\in\{1,\dots, d\}$,
$$
\frac{\gamma_{X}^{s}(0)-\gamma_{X}^{s}(\eps  e_{j})  }{ |\eps| } \leq \frac{1}{2} V^s_{e_j}(X),\quad \eps\neq 0,
$$
and
$$
\Lip_{j}(\gamma_{X}^{s})=\Lip_{j}( \gamma_{X}^{s},0)
= \lim_{\eps \rightarrow 0} \frac{\gamma_{X}^{s}(0) - \gamma_{X}^{s}(\eps e_j)}{|\eps|}
= \frac{1}{2} V^s_{e_j}(X).
$$
\end{proposition}

The proof of the proposition is an adaptation of similar results for covariogram functions~\cite{Galerne_computation_perimeter_covariogram_2011}.
We first state a lemma regarding local covariogram of deterministic sets.

\begin{lemma}
For all $y, z\in\R^d$, $W\in\W$, and $A\in\M$,
\begin{equation}
\label{eq:difference_local_covariogram_less_than_in_zero_translated}
\delta_{y;W}(A) - \delta_{z;W}(A)
\leq \delta_{0;-y+W}(A) - \delta_{z-y;-y+W}(A).
\end{equation}
\end{lemma}

\begin{proof}
We have 
$$
\begin{aligned}
\delta_{y;W}(A) - \delta_{z;W}(A)
& = \L^d(A\cap (y+A)\cap W) - \L^d(A\cap (z+A)\cap W)
\\
& \leq \L^d((A\cap (y+A)\cap W) \setminus (A\cap (z+A)\cap W))
\\
& \leq \L^d(((y+A)\cap W) \setminus ((z+A)\cap W))
\\
& \leq \L^d((y+A)\cap W) - \L^d((y+A)\cap (z+A)\cap W)
\\
& \leq \L^d(A\cap (-y+W)) - \L^d(A\cap (z-y+A)\cap (-y+W))
\\
& \leq \delta_{0;-y+W}(A) - \delta_{z-y;-y+W}(A).
\end{aligned}
$$
\end{proof}

\begin{proof}[Proof of Proposition~\ref{prop:specific_covariogram_directional_lipschitz_constant}]

Let us first check that $\gamma_{X}^{s}$ is even.
For all $y\in\R^d$, 
$$
\gamma_{X}^{s}(-y)
= \E\left( \L^d(X\cap (-y+X) \cap (0,1)^d \right)
= \E\left( \L^d( (y+X)\cap X \cap (y+(0,1)^d) \right)
= \gamma_{X}^{s}(y).
$$
Let us turn to the inequality.
As a direct consequence of~\eqref{eq:difference_local_covariogram_less_than_in_zero_translated},
$$
\gamma_X(y;W) - \gamma_X(z;W) \leq \gamma_X(0;-y+W) - \gamma_X(z-y;-y+W).
$$
But, since $\gamma_X(y;W) = \gamma^s_X(y) \L^d(W)$, 
$
\gamma_{X}^{s}(y)-\gamma_{X}^{s}(z) \leq \gamma_{X}^{s}(0) - \gamma_{X}^{s}(z-y),
$
and interchanging $y$ and $z$ yields  $|\gamma_{X}^{s}(y)-\gamma_{X}^{s}(z)| \leq \gamma_{X}^{s}(0) - \gamma_{X}^{s}(y-z)$.
This inequality yields
$$
\Lip_{j}(\gamma_{X}^{s})
= \sup_{y\in\R^d,~\eps \in\R} \frac{|\gamma_{X}^{s}(y+\eps e_j)-\gamma_{X}^{s}(y)|}{|\eps |}
= \sup_{\eps\in\R} \frac{\gamma_{X}^{s}(0)-\gamma_{X}^{s}(\eps e_j)}{|\eps|}=\Lip_{j}( \gamma _{X}^{s},0).
$$
Since $\gamma_{X}^{s}$ is even, for all $\eps\neq 0$,
$$
\begin{aligned}
& \frac{\gamma_{X}^{s}(0)-\gamma_{X}^{s}(\eps   e_{j})  }{ |\eps|}
\\
& = \frac{1}{2} \frac{\gamma_{X}^{s}(0)-\gamma_{X}^{s}(\eps   e_{j})+\gamma_{X}^{s}(0)-\gamma_{X}^{s}(-\eps   e_{j})}{ |\eps|}
\\
& = 
\frac{1}{2} \sup_{c>0}\frac{\gamma_{X}^{s}(0)-\gamma_{X}^{s}(\eps  e_{j})+\gamma_{X}^{s}(0)-\gamma_{X}^{s}(-\eps  e_{j})}{ |\eps|}\frac{c^{d}-|\eps|   c}{c^{d}}
\\
& = \frac{1}{2}\sup_{c>0}\frac{ (\gamma_{X}^{s}(0)-\gamma_{X}^{s}(\eps  e_{j}))\L^d((0,c)^d\ominus[-\eps e_j,0]) + (\gamma_{X}^{s}(0)-\gamma_{X}^{s}(-\eps  e_{j})) \L^d((0,c)^d\ominus[0,\eps  e_j])}{|\eps| \leb((0,c)^{d})}
\\
& = \frac{1}{2} \sup_{c>0} \sigma_X(\eps e_{j} ; (0,c)^{d}) \frac{1}{ \leb((0,c)^{d})}
\\
& \leq \frac{1}{2} \sup_{c>0} \E V_{e_j}(X;(0,c)^{d}) \frac{1}{ \leb((0,c)^{d})} = \frac{1}{2} V_{e_{j}}^{s}(X),
\end{aligned}
$$
where   the inequality follows from~\eqref{eq:approx-perimeter}.
This shows that
$$
\Lip_j(\gamma_{X}^{s},0) 
= \sup_{\eps \in \mathbb{R}}\frac{\gamma_{X}^{s}(0)-\gamma_{X}^{s}(\eps   e_{j})  }{ | \eps | } 
\leq \frac{1}{2} V_{e_{j}}^{s}(X).
$$
Besides, for all $\eps \neq 0$ and $c>0$,
$$
\frac{1}{2} \sigma_X(\eps e_{j} ; (0,c)^{d}) \frac{1}{ \leb((0,c)^{d})}
\leq \frac{\gamma_{X}^{s}(0)-\gamma_{X}^{s}(\eps   e_{j})  }{ |\eps|}
\leq \frac{1}{2} V_{e_{j}}^{s}(X),
$$
and according to Corollary~\ref{cor:tcd-per}, the left-hand term tends to $\frac{1}{2} V_{e_{j}}^{s}(X)$ when $\eps\to 0$.
Hence,
$$
\lim_{\eps\rightarrow 0} \frac{\gamma_{X}^{s}(0)-\gamma_{X}^{s}(\eps   e_{j})  }{ |\eps|}
= \frac{1}{2} V_{e_{j}}^{s}(X)
= \sup_{\eps \in \mathbb{R}}\frac{\gamma_{X}^{s}(0)-\gamma_{X}^{s}(\eps   e_{j})  }{ | \eps | }.
$$

\end{proof}

\subsection{Anisotropic perimeter approximation for pixelized sets}

The proofs of our main results rely on several approximations involving pixelized sets and discretized covariograms.
We proved in the previous section that the directional variations as well as the anisotropic perimeter can be computed from limits of difference quotients at zero of the local covariogram.
We show here that for pixelized sets, the anisotropic perimeter $\Per_{\mathbf{B}}(A; W)$ can be expressed as a finite difference at zero of the local covariogram functionals.

For $n\in\N^* = \N\setminus\{0\}$, we consider the pixels $C^n_k = \frac{1}{n} k +  \left[0,n^{-1}\right]^d, k\in \mathbb{Z}^{d}$, 
that are the cells of the lattice $n^{-1} \Z^d$.
Denote by $\sM_n$ the algebra of $\R^d$ induced by the sets $C_{k}^{n}$, $k\in \Z^d$,
and denote $\W_{n}=\W\cap \M_{n}$.
For any $W\in \W_{n}$, also denote $\M_n(W) = \{A\in\M_n:~A\subset W\}$ the sets of pixelized sets contained inside $W$.
Remark that for any set $A\in \M_{n}$, there is a unique subset $I_{A}$ of $\mathbb{Z}^{d}$ such that $A$ is equivalent in measure to
$\cup _{k\in I_{A}}C_{k}^{n}$.

\begin{proposition}
\label{prop:per_approx_local_covariogram}
Let $\mathbf{B} = (e_1,e_2,\dots,e_d)$ be the canonical basis of $\R^d$ and let $n\in\N^*$.
For all $A\in\sM_n$, $W\in \W_{n}$, and $j\in\{1,\dots,d\}$,
$$
V_{e_j}(A;W) =  \sigma _{n^{-1}e_{j};W}(A) .
$$
Hence, for all $A\in\sM_n$ and $W\in  \W_{n}$,
$\Per_{\mathbf{B}}(A ; W)=  \sum_{j=1}^{d}\sigma _{n^{-1}e_{j};W}(A).$
\end{proposition}

\begin{proof}
Let $0<\varepsilon \leq n^{-1}$. Consider the quantity
$$
\begin{aligned}
\delta _{\varepsilon e_{j};W\ominus [-\varepsilon e_{j},0]}(A)&=\leb(A\cap (\varepsilon e_{j}+A)\cap (W\ominus [-\varepsilon e_{j},0] ))\\
&=\leb\left( \left( \bigcup _{k\in I_{A}} C_{k}^{n}\right)\cap \left( \bigcup _{l\in I_{A}}(\varepsilon e_{j}+C_{l}^{n}) \right)\cap (W\ominus [-\varepsilon e_{j},0] )\right).
\end{aligned}
$$
The two unions are over sets with pairwise negligible intersection, whence 
$$
\begin{aligned}
\delta _{\varepsilon e_{j};W\ominus [-\varepsilon e_{j},0]}(A)&=\sum_{k,l\in I_{A}}\leb(C_{k}^{n}\cap (C_{l}^{n}+\varepsilon e_{j})\cap ( W\ominus [-\varepsilon e_{j},0] )).
\end{aligned}
$$
Since $0<\varepsilon \leq n^{-1}$, for $k,l\in I_{A}$, 
$$
\begin{aligned}
\leb(C_{k}^{n}\cap (C_{l}^{n}+\varepsilon e_{j})\cap ( W\ominus [-\varepsilon e_{j},0] ))=
\begin{cases}
n^{-(d-1)}(n^{-1}-\varepsilon ) & \text{if $l=k\in I_{W}$},\\
\varepsilon n^{-(d-1)} & \text{if $l=k-e_{j}$ and $k,l\in I_{W}$},\\
0  & \text{otherwise}.
\end{cases}
\end{aligned}
$$
These assertions are straightforward,  one simply has to be cautious in the case $l=k-e_{j},k\in I_{W},l\notin I_{W}$, contribution of which is $0$. 
Summing up those contributions and doing similar computations for the quantities $\delta _{0;W\ominus [-\varepsilon e_{j},0]}(A),\delta _{-\varepsilon e_{j};W\ominus [0,\varepsilon e_{j} ]}(A),$ and $\delta _{0;W\ominus [0,\varepsilon e_{j} ]}(A)$ yields that, for some real numbers $\alpha ,\beta $ independent of $\varepsilon $, for all $\varepsilon \in (0, n^{-1}]$, 
$$
\begin{aligned}
\sigma _{\varepsilon e_{j};W}(A)=\frac{\alpha}{\varepsilon} + \beta.
\end{aligned}
$$
Proposition~\ref{prop:per-approx} then implies $\alpha =0$, $\beta =V_{e_{j}}(A;W)$, which yields the desired conclusion with $\varepsilon = n^{-1}$.
\end{proof}

\section{Realisability result}

\label{sec:proof}

In this section we explicit some considerations related to our realisability result and give its proof.

\subsection{Realisability problem and regularity modulus}
\label{sec:main-proof}

Recall that the local covariogram of a RAMS $X$ is $\gamma_{X}(y;W)=\E \delta _{y;W}(X)$.
Introduce a regularized realisability problem for local covariogram.
Put $U_n=(-n,n)^{d}$. 
Define the \textit{weighted anisotropic perimeter} by
\begin{equation*}
\PerB^{\beta }(A) = \sum_{n\geq 1}\beta_{n} \PerB(A;U_{n}),
\end{equation*}
where the sequence $(\beta_n)$ is set to $\beta_n = 2^{-n} (2n)^{-d}$ so that $\sum_{n\geq 1} \beta_n\L^d(U_n) = 1$.
For a given function $\gamma:\R^d\times\W \rightarrow \R$, define 
$$
\sigma_{\gamma } (u;W)
=
\frac{1}{\|u\|}[
\gamma(0;W\ominus[-u,0]) -  \gamma(u;W\ominus[-u,0])\\
+ \gamma(0;W\ominus[0,u]) - \gamma(-u;W\ominus[0,u])].
$$
Define for all windows $W\in \W$ the constant
$L_j(\gamma, W)\in [0,+\infty]$ by
\begin{equation}
\begin{aligned}
L_{j}(\gamma ,W)=\sup_{\eps \in \mathbb{R}}\sigma _{\gamma }(\eps e_{j};W),\quad j\in\{1,\dots,d\}.
\label{eq:definition_Lj_constants_as_sup}
\end{aligned}
\end{equation}
$L_{j}(\gamma ,W)$ is related to the Lipschitz property of $\gamma $ in its spatial variable.
The motivation for considering this particular constant comes from
Corollary~\ref{cor:tcd-per}
which shows that if $\gamma_X$ is the local covariogram of a RAMS $X$, then
$$
\begin{aligned}
\E V_{e_{j}}(X;W) &= \sup_{\eps\in\R} \sigma _{\gamma _{X}}(\eps e_ j;W ).
\end{aligned}
$$

\begin{theorem}
\label{thm:realisability_over_Rd_perBbeta}
Let $\gamma:\R^d\times \W \rightarrow\R$ be a  function and $r\geq 0$.
Then $\gamma$ is realisable by a RAMS $X$ such that 
\begin{equation*}
\E \PerB^{\beta }(X)\leq r
\end{equation*}
if and only if $\gamma  $ is admissible (see Definition~\ref{def:covariogram_admissible_measurable}) and 
\begin{equation}
\label{eq:avg-perimeter}
\sum_{n\geq 1}\beta_{n}\left(\sum_{j=1}^d L_j (\gamma  , U_{n})\right)\leq r,
\end{equation}
where for all $j\in\{1,\dots,d\}$ and $n\geq 1$, the constant $L_j (\gamma, U_{n})$ is defined by~\eqref{eq:definition_Lj_constants_as_sup}.
\end{theorem}

%
 
The stationary counterpart of the above theorem is stated and proved in Section~\ref{subsubsec:proof_realisability_Rd_stationary}. 
Let us specialize a general definition from~\cite[Def.~2.5]{LacMol} to our framework.

\begin{definition}[Regularity moduli]
Let $G$ be a vector space of measurable real functions on $\M$.
A $G$-\emph{regularity modulus} on $\sM $  is a lower semi-continuous function
$\chi: \M \mapsto [0,+\infty]$ such that, for all $g\in G$,
the level set
$$
H_g =\left\{ A\in\M :~ \chi(A)\leq g(A) \right\} \subset \sM 
$$
is relatively compact for the convergence in measure.
\end{definition}

We give the following result, straightforward consequence of Proposition 2.2 and Theorem 2.6 in~\cite{LacMol} for bounded continuous functions, see in particular the discussion after the proof of Theorem 2.6.

\begin{theorem}[Lachi\`eze-Rey - Molchanov~\cite{LacMol}]
\label{thm_lacmol_realisability}
Let $G$ be a vector space of real continuous bounded functions  on $\M$ that comprises constant functions. Let $\chi $ be a ${G}$-regularity modulus, and $\Phi $ be a linear function on $G$ such that $\Phi (1)=1$. Then, for any given $r\geq 0$, there exists a RAMS $X\in \M$ such that 
\begin{equation}
\label{eq:real-pb-general}
\begin{cases}
\E g(X)=\Phi (g),\quad g\in G,\\
\E \chi (X)\leq r
\end{cases}
\end{equation}
if and only if
\begin{equation}
\label{eq:sup-inf-condition}
\sup_{g\in G}\inf_{A\in \M}\chi (A)-g(A)+\Phi (g)\leq r.
\end{equation}
\end{theorem}

In our setting, call $G$ the vector space generated by the constant functionals and the local covariogram functionals $A\mapsto \delta_{y;W}(A)$, $y\in\R^d$, $W\in\W$.

\begin{proposition}
$\PerB^\beta$ is a $G$-regularity modulus (and therefore a $G^{*}$-regularity modulus for any subspace $G^{*}\subset G$).
\end{proposition}

\begin{proof}
By definition of a regularity modulus, one has to show that the $\PerB^\beta$-level sets are relatively compact.
Consider a sequence $(A_n)$ such that $\PerB^\beta(A_n)\leq c$ for all $n\in\N$.
Then, for all $n, m\in\N$, $\PerB(A_n;U_m) \leq \frac{c}{\beta_m} <\infty$, and thus $(A_n)$ is a sequence of sets of locally finite perimeter whose perimeter in any open bounded set $U\subset\R^d$ is uniformly bounded.
According to~\cite[Theorem 3.39]{Ambrosio_Fusco_Pallara_functions_of_bounded_variation_free_discontinuity_problems_2000}, there exists a subsequence of $(A_n)$ that locally converges in measure in $\R^d$.
\end{proof}

For $g\in G$, denote by $\dom(g)$ the smallest open set   such that for every measurable set $A$, $g(A)=g(A\cap \dom(g))$.
If $g$ has the form 
\begin{equation}
\label{eq:g-representation}
g=\sum_{i=1}^{q}a_{i}\delta _{y_{i};W_{i}},
\end{equation}
we have $\dom(g)\subset \cup _{i}(W_{i}\cup (-y_{i}+W_{i}))$,  there is not equality because such a decomposition is not unique.

We can  turn to the proof of Theorem~\ref{thm:realisability_over_Rd_perBbeta}.
It involves several technical lemmas that are stated within the proof when needed.
Their demonstrations are delayed to the end of the section.

\begin{proof}[Proof of Theorem~\ref{thm:realisability_over_Rd_perBbeta}]

\textsl{Necessity:}
If $X$ is a RAMS, then the admissibility of $\gamma_{X}$ is the consequence of the positivity of the mathematical expectation, see the discussion below Definition~\ref{def:covariogram_admissible_measurable}.
According to Proposition~\ref{prop:per-approx},  for all $n\geq 1$, with probability $1$,$$
\sigma _{\eps e_{j};U_{n}}(X) \leq V_{e_j}(X;U_{n}).
$$
After taking the expectation, the supremum of the left hand member over $\eps >0$ is $L_{j}(\gamma ,W)$.
Summing over $j$ yields $\sum_{j=1}^{d}L_{j}(\gamma , U_{n})\leq \PerB(X ; U_n)$, and multiplying by $\beta_n$ and summing over $n$ yields \eqref{eq:avg-perimeter}.

\smallskip

\noindent\textsl{Sufficiency:}
Call $G_n\subset G$ the set of functionals $g = c + \sum_{i=1}^q a_i \delta_{y_i;W_i}$ such that for all $i$, $y_i\in n^{-1}\Z^d$,
and $W_i\in \W_{n}$, \emph{i.e.} the closures of the $W_i$ are pixelized sets.
Denote by $G^\ast = \bigcup_{n\geq 1} G_n$.
Remark that each $G_n$ is a vector space and that $G^\ast\subset G$ is a vector space as well:
Indeed, if $g_1\in G_n$ and $g_2\in G_m$, then  $g_1+g_2\in G_{mn}$.
To apply Theorem~\ref{thm_lacmol_realisability} to $G^{*}$ we need to show that,
$$
\sup_{n\geq 1} \sup_{g\in G_n} \inf_{A\in\M} \PerB^{\beta}(A) - g(A) + \Phi(g) \leq r,
$$
where $\Phi$ is defined by 
$$
\Phi\left(c + \sum_{i=1}^{q} a_{i}\delta _{y_{i};W_{i}}\right) = c + \sum_{i=1}^{q} \gamma(y_i, W_i).
$$
First remark that $\Phi $ is a positive operator because $\gamma  $ is admissible, see Definition~\ref{def:covariogram_admissible_measurable}. Let $g\in G_n$. Define $m_{g}= \inf_{A\in\M} \PerB^{\beta}(A) - g(A) + \Phi(g).$
Let $p\in \N$ large enough such that $\dom(g)\subset (-p,p)^d$.
For all $c>0$ denote $\M_{n}^{c}=\M_{n}((-c,c)^{d})$.
We have
\begin{equation*}
m_{g}  \leq  \inf_{A\in\M_{n}^{p }} \PerB^{\beta}(A) - g(A) + \Phi(g)
\end{equation*}
because $\M_{n}^{p }\subset \M$.
The proof is based on an approximation of the perimeter by a discretized functional with compact domain, summarized by  the following lemma.

\begin{lemma}
\label{lem:weighted_perimeter_gnp_functional-bis}
For $n,p\geq 1$, put $U_n^p =(-p-1/n,p+1/n)^{d}$. There exists $g_{n,p}\in G_{n}$ with $\dom(g_{n,p})\subset U_n^p $ such that 
\begin{equation*}
g_{n,p}(A)=\PerB^{\beta}(A)
\end{equation*}
for all $A\in\M_{n}^{p}$.
Its explicit expression is
$$
\begin{aligned}
  g_{n,p}(A) =\sum_{m = 1}^{p} \beta_{m} 
\sum_{j=1}^d  \sigma _{n^{-1}e_{j};U_{n}}(A)+ \left(\sum_{m = p+1}^{+\infty} \beta_{m}\right) 
\sum_{j=1}^d \sigma _{n^{-1}e_{j};U_{p}}(A).\end{aligned}
$$
Furthermore, for all $A\in \M_{n}$, 
\begin{equation*}
| g_{n,p}(A)-g_{n,p}(A\cap (-p,p)^{d}) |\leq E_{n,p},
\end{equation*}
where $E_{n,p}=8 d n 2^{-p}(p+1)^{-d} \left( \left(p+\tfrac{1}{n}\right)^d - p^d  \right)$.
\end{lemma}

Therefore, $g_{n,p}= \PerB^{\beta }$ on $\M_{n}^{p}$, and
$$
\begin{aligned}
m_{g}&\leq  \inf_{A\in \M_{n}^{p }  } g_{n,p}(A)-g(A)+\Phi (g)\\
&\leq  \inf_{A\in \M_{n}^{p+1/n}  } g_{n,p}(A)-g(A)+\Phi (g)+E_{n,p},
\end{aligned}
$$
because $g(A)=g(A\cap \dom(g))=g(A\cap (-p,p)^{d})=g(A\cap (-p-1/n,p+1/n)^{d})$,
and
$| g_{n,p}(A)-g_{n,p}(A\cap (-p,p)^{d}) | \leq E_{n,p} $,
where the error term $E_{n,p}$ is computed in Lemma~\ref{lem:weighted_perimeter_gnp_functional-bis}.
We  need the following lemma, also proved afterwards.
 
\begin{lemma}
\label{lem:extrema_local_covariogram}
Any functional $g\in G_{n}$ reaches its infimum on an element of $\M_{n}(\dom(g))$.
\end{lemma}

We have  $\dom(g_{n,p}-g)\subset U_n^p $, whence  by Lemma~\ref{lem:extrema_local_covariogram}, $g_{n,p}-g$  reaches its  infimum over $\M$ on $\M_{n}^{p+1/n}$, and 
\begin{equation*}
 \inf_{A\in \M_{n}^{p+1/n}  } g_{n,p}(A)-g(A)  = \inf_{A\in \M   } g_{n,p}(A)-g(A) =\inf_{A\in \M}(g_{n,p}-g)(A)\leq \Phi (g_{n,p}-g),
\end{equation*} where the last inequality is a consequence of the positivity of $\Phi $. Therefore,
\begin{equation}
\label{eq:1}
m_{g} \leq \Phi (g_{n,p}-g)+\Phi (g)+E_{n,p}=\Phi (g_{n,p})+E_{n,p}.
\end{equation}
Let us bound $\Phi(g_{n,p})$.
Recall that by definition $\Phi(\delta_{y;W}) = \gamma(y;W)$.
The definition of the constants $L_j(\gamma, W)$ and the expression of $g_{n,p}$ yield
$$
\Phi(g_{n,p}) \leq \sum_{m=1}^p \beta_m \left(\sum_{j=1}^d L_j(\gamma, U_m)\right) + \left(\sum_{m = p+1}^{+\infty} \beta_{m}\right) \left(\sum_{j=1}^d L_j(\gamma,  U_n^p )\right).
$$ 
\begin{lemma}
\label{lem:Lj_constants_are_increasing}
For all admissible functions $\gamma$, $L_j(\gamma;W) \leq L_j(\gamma,W')$ for any $j\in\{1,\dots,d\}$ and $W, W'\in\W$ such that $
W\subset W'.$
\end{lemma}

By Lemma~\ref{lem:Lj_constants_are_increasing}, since $\gamma$ is admissible, and for all $m>p$, $ U_n^p  \subset U_m$,
$$
\left(\sum_{m = p+1}^{+\infty} \beta_{m}\right) \left(\sum_{j=1}^d L_j(\gamma, U_n^p )\right)
= \sum_{m = p+1}^{+\infty} \beta_{m}\left(\sum_{j=1}^d L_j(\gamma, U_n^p )\right) 
\leq \sum_{m = p+1}^{+\infty} \beta_{m}\left(\sum_{j=1}^d L_j(\gamma, U_m)\right).
$$
Hence,
$$
\Phi(g_{n,p}) \leq \sum_{m=1}^{+\infty} \beta_m \left(\sum_{j=1}^d L_j(\gamma, U_m)\right) \leq r.
$$
Coming back to \eqref{eq:1} yields
\begin{equation*}
m_{g}\leq r + E_{n,p}.
\end{equation*}
Since for all $n\geq 1$, $E_{n,p}$ tends to $0$ as $p$ tends to $+\infty$, one has
$m_{g}\leq r$.
Since $n\geq 1$ and $g\in G_n$ were arbitrarily chosen, we conclude that
$$
\sup_{g\in G^\ast} \inf_{A\in\M} \PerB^{\beta}(A) - g(A) + \Phi(g) \leq r.
$$
Hence, we can apply Theorem~\ref{thm_lacmol_realisability} that ensures that there exists a RAMS $X$ solution of the problem~\eqref{eq:real-pb-general}.
This RAMS $X$ satisfies $\E\PerB^{\beta }(X)\leq r$ and, for $y$ in the set $\Q^{d}$ of vectors with rational coordinates,
\begin{equation*}
W \in \W\cap \bigcup_{n\in\N^*} \M_n,\quad  \gamma_{X}(y;W) = \gamma(y;W).
\end{equation*}
It only remains to show that this equality between $\gamma_X$ and $\gamma$ extends to all couple $(y;W)\in \R^d\times \W$
using the continuity of both $\gamma_X$ and $\gamma$.

First, regarding the $W$-variable, since $\gamma_X$ and $\gamma$ are both admissible, and by Proposition~\ref{prop:continuity_local_covariogram}
$|\delta_{y;U}(A) - \delta_{y;W}(A)| \leq \L^d(U\Delta W)$ for all $U, W \in\W$,
both $\gamma_X$ and $\gamma$ are continuous with respect to the convergence in measure.
Besides, the set of pixelized sets $\W\cap \bigcup_{n\in\N^*} \M_n$ is dense in $\W$ for the convergence in measure.
Indeed, given $W\in\W$, one easily shows by dominated convergence that the sequence $W_n = \bigcup_{k\in\Z^d}\left\{ C^n_k:~ C^n_k\subset W \right\}$ converges in measure towards $W$, since due to the hypothesis $\leb(\partial W)=0$, for almost all $x\in\R^d$ either
$x\in W$ or $x\in\R^d\setminus \overline{W}$, where $\overline{W}$ denotes the closure of $W$.

Regarding the $y$-variable, $y\mapsto \gamma_X(y;W) = \E \delta_{y;W}(X)$ is continuous since $y\mapsto \delta_{y;W}(X)$ is a.s. continuous and bounded by $\L^d(W)$.
To conclude the proof, let us show that $y\mapsto \gamma(y;W)$ is also continuous over $\R^d$, which is the purpose of the following lemma.

\begin{lemma}
\label{lem_admissible_Delta_are_locally_lipschitz}
Let $\gamma$ be an admissible function.
Let $y\in\R^d$, $W\in\W$, and $r>0$.
For all $z\in\R^d$ and $\rho>0$, denote by $C(z,\rho)$  the hypercube of center $z$ and half size length $\rho$, 
that is $C(z,\rho) = \{z'\in\R^d:~\|z'-z\|_{\infty }\leq \rho\}$.

Then, for all $z,z'\in C(y,r)$,
$$
|\gamma(z;W) - \gamma(z';W)| \leq \sum_{j=1}^d L_j(\gamma, W \oplus C(-y,3r)) |z_j' - z_j|.
$$
In particular, if $\gamma$ satisfies~\eqref{eq:avg-perimeter}, then for all $W\in\W$, the map $y\mapsto \gamma(y;W)$ is locally Lipschitz.
\end{lemma}

\end{proof}
 
We  turn to the proofs of the lemmas.

\begin{proof}[Proof of Lemma~\ref{lem:weighted_perimeter_gnp_functional-bis}]

First remark that for all sets $A$ of finite perimeter such that $A\subset (-p,p)^d$ for some integer $p\geq 1$ and for all $m>p$, since $A\cap (U_m \setminus (-p,p)^d) = \emptyset$,
$$
\PerB(A;U_m)  = \PerB(A;(-p-\tfrac{1}{n},p+\tfrac{1}{n})^d) = \PerB(A; U_n^p).
$$
Consequently,
$$
\PerB^{\beta}(A)
= \sum_{m\geq 1}\beta_{m} \PerB(A;U_m)
= \sum_{m = 1}^{p} \beta_{m} \PerB(A;U_m) + \left( \sum_{m = p+1}^{+\infty} \beta_{m}  \right) \PerB(A;U_n^p).
$$
According to Proposition~\ref{prop:per_approx_local_covariogram}, for all pixelized sets $A\in\M_n^p$,
all the perimeters $\PerB(A;U_m)$, $1\leq m\leq p$, and $\PerB(A;U_n^p)$, can be expressed as some linear combination of local covariograms, which gives
$$
\begin{aligned}
  \PerB^{\beta}(A) =  
\sum_{m = 1}^{p} \beta_{m} 
\sum_{j=1}^d \sigma_{n^{-1}e_{j};U_{m}}(A)  + \left(\sum_{m = p+1}^{+\infty} \beta_{m}\right) 
\sum_{j=1}^d \sigma_{n^{-1}e_{j};U_n^{p}}(A) .\end{aligned}
$$
The linear combination on the right-hand side is an element of $G$ that will be denoted $g_{n,p}$ in what follows.
Note that $\dom(g_{n,p}) \subset  U_n^{p} = (-p-\frac{1}{n},p+\frac{1}{n})^d$.
It remains to show the inequality $| g_{n,p}(A)-g_{n,p}(A\cap (-p,p)^{d}) |\leq E_{n,p}$.
Using Proposition~\ref{prop:continuity_local_covariogram},
for all $A, B\in\M$, $j\in\{1,\dots,d\}$, and $W\in\{ U_1,\dots, U_p, U_n^{p}\}$,
$$
|\sigma_{n^{-1}e_{j};W}(A) - \sigma_{n^{-1}e_{j};W}(B)|
\leq 8 n \L^d((A\Delta B) \cap W).
$$
Hence, for all $A\in\M_n$,
$$
\begin{aligned}
|g_{n,p}(A) - g_{n,p}(A \cap (-p,p)^d)|
& = |g_{n,p}(A \cap U_n^{p}) - g_{n,p}(A \cap (-p,p)^d)|
\\
& \leq \left(\sum_{m = p+1}^{+\infty} \beta_{m}\right) d 8 n \L^d((A \cap U_n^{p}) \Delta (A\cap (-p,p)^d)),
\end{aligned}
$$
since for all $m\in\{1,\dots,p\}$, $((A \cap U_n^{p}) \Delta (A\cap (-p,p)^d))\cap U_m = \emptyset$.
For all $m\geq p+1$,  $\beta_m = 2^{-m} (2m)^{-d} \leq 2^{-m} (2(p+1))^{-d}$.
Hence, 
$$
  \sum_{m = p+1}^{+\infty} \beta_{m} \leq 2^{-d} (p+1)^{-d} \sum_{m = p+1}^{+\infty} 2^{-m} = 2^{-d-p}(p+1)^{-d}.
$$
Besides, $\L^d((A \cap U_n^{p}) \Delta (A\cap (-p,p)^d)) \leq \L^d(U_n^{p}\setminus (-p,p)^d) = 2^d \left( \left(p+\tfrac{1}{n}\right)^d - p^d  \right)$.
Finally
$$
\begin{aligned}
|g_{n,p}(A \cap U_n^{p}) - g_{n,p}(A \cap (-p,p)^d)|
& \leq 8 d n 2^{-p}(p+1)^{-d} \left( \left(p+\tfrac{1}{n}\right)^d - p^d  \right).
\end{aligned}
$$

\end{proof}

\begin{proof}[Proof of Lemma~\ref{lem:extrema_local_covariogram}] 

Put $W=\dom(g)\in \W_{n}$. Then $g(A)=g(A\cap W)$ for any $A\in \M$. 
Now that the problem is restricted to the bounded pixelized domain $W$, it remains to show that the extrema of $g$ on $\M(W)$ are reached by sets of $\M_n(W)$.
Let us turn into the details.

Without loss of generality, assume that $g$ has the form
$$
\begin{aligned}
g=\sum_{i=1}^{q}a_{i}\delta _{y_{i};W_{i}},
\end{aligned}
$$
for some $y_{i}\in n^{-1}\mathbb{Z} ^{d},W_{i}\in \W_{n},a_{i}\in \mathbb{R}$.
Denote by $I_n(W)$ the set of all indexes $k\in\Z^d$ such that the hypercube $C^n_k$ is included in $\overline{W}$, we then also have $\overline{W}=\cup _{k\in I_{n}(W)}C_{k}^{n}$.
For $A\in\sM(W)$, $n \geq 1$,  denote by $A^n_{k}=A\cap C_k^n,  k\in I_n(W)$,
the intersection of $A$ with the hypercube $C^n_{k}$, and by 
$\tilde A_{k}^{n}=-k+nA^{n}_{k}$
its rescaled translated version comprised in $[0,1]^d$.
Consider the probability space $(\Omega=[0,1)^d, \mathcal{A} = \B([0,1)^d),\mathbb{P}=\leb)$, on which we define the $\{0,1\}^{I_n(W)}$-valued random vector
$
\begin{array}{rcl}
Y^{A}(w)= (\Ind_{\tilde{A}^n_k}\left( \omega\right))_{k\in I_n(W)},\omega \in \Omega .
\end{array}
$
The measures of the pairwise  intersections $\L^d\left(\tilde{A}^n_k\cap \tilde{A}^n_{l}\right)$ can thus be seen as  the components of the covariance matrix $C(A)= \left(C _{k,l}(A)\right)_{k,l\in I_n(W)}$ of the random vector $Y^{A}$, i.e.
$$
C_{k,l}(A) = \mathbb{E}\left( Y^{A}_{ k} Y^{A}_{ l} \right) = \mathbb{E}\left( \Ind\left(\omega \in \tilde A^n_k \right) \Ind\left( \omega \in \tilde  A^n_l \right) \right) = \L^d\left(\tilde{A}^n_k\cap \tilde{A}^n_{l}\right),
~k,l\in I_n(W).
$$
Let us prove that $g(A)$ can be written as
$$
g(A) = \sum_{k,l\in I_n(W)} \beta_{k,l} \leb(\tilde{A}_{k}^{n}\cap \tilde A_{l}^{n})
$$
for some coefficients $\beta =(\beta _{k,l})_{k,l\in I_n(W)}$ depending solely on $g$.
Putting  $k_{i} = n y_i\in \mathbb{Z} ^{d}$, we have
$$
\begin{aligned}
g(A)
& = \sum_{i=1}^{q}a_{i} \delta_{n^{-1}k_{i}, W_i} (A) \\
& = \sum_{i=1}^{q}a_{i}\L^d(A\cap (n^{-1}k_{i}+A)\cap  {W_i})\\
& = \sum_{i=1}^{q}a_{i}\sum_{k\in I_n(W)}\Ind(C_{k}^{n}\subset \overline{W_i})\int_{C_{k}^{n}}\Ind (x\in A,x\in n^{-1}k_{i}+A )dx\\
& = \sum_{i=1}^{q}a_{i}\sum_{k,l\in I_n(W)}\Ind (l=k-k_{i} )\Ind(C_{k}^{n}\subset \overline{W_i})\int_{C_{k}^{n}}\Ind ( x\in A,x\in n^{-1} (k-l)+A)dx\\
& = \sum_{k,l\in I_n(W)}\sum_{i=1}^{q}a_{i}\Ind (l=k-k_{i} )\Ind(C_{k}^{n}\subset \overline{W_i})\int_{C_{0}^{n}}\Ind (x\in -n^{-1}k+A,x\in -n^{-1}l+A)dx\\
& = \sum_{k,l\in I_n(W)} \underbrace{n^{-d} \left(\sum_{i=1}^{q}a_{i}\Ind (l=k-k_{i} )\Ind(C_{k}^{n}\subset \overline{W_i})\right)}_{= \beta_{k,l}} \L^d(\tilde A_{k}^{n}\cap \tilde A_{j}^{n}).
\end{aligned}
$$
Then, one can write
$$
g (A) = \sum_{k, l\in I_n(W)} \beta_{k,l} C_{k,l}(A) = \langle \beta, C (A) \rangle,
$$
where $\langle\cdot,\cdot\rangle$ stands for the classical scalar product between matrices.
Denote by $\Gamma_{n}$ the set of covariance matrices of all random vectors having values in $\{0,1\}^{I_n(W)}$.
Since for every set $A$ one can associate some covariance matrix $C (A)$ such that $g (A) = \langle \beta,  C (A)\rangle$,
one can write
$$
\inf_{A\in \M(W)} g(A)\geq \inf_{C\in\Gamma_{n}} \langle \beta, C  \rangle.
$$
The optimization problem on the right-hand side of this inequality is a linear programming problem on the bounded convex set $\Gamma_{n}$.
Hence we are ensured that there exists an optimal solution $C^\ast$ of this problem which is an extreme point of $\Gamma_n$.
As shown in \cite[Th. 2.5]{Lac13b}, the extreme points of $\Gamma_n$  are covariance matrices associated with deterministic random vectors, see also \cite{DezLau97}, where $\Gamma _{n}$ is called the \emph{correlation polytope} and studied more deeply. That is there exists a fixed vector $z^\ast\in\{0,1\}^{I_n(W)}$ such that $C^\ast_{k,l} = z^\ast_k z^\ast_l$ minimizes $\langle \beta, C  \rangle$.
Given this vector $z^\ast\in\{0,1\}^{I_n(W)}$, define the set $A^\ast$ as the union of the hypercubes 
\begin{equation*}
A^\ast =\bigcup_{\{k:~z_{k}^\ast =1\}} C_{k}^n\cap W.
\end{equation*}
Then one sees that the covariance matrix $C(A^\ast )$ associated with the deterministic set $A^\ast$ is equal to $C^\ast$.
Furthermore it is clear that $A^\ast$ is measurable with respect to the $\sigma$-algebra generated by the $C_{k}^n$, meaning exactly $A^\ast\in \M_{n}(W)$.
Hence we have shown that,
$$
\inf_{A\in \M(K)} g(A) \geq \inf_{C\in\Gamma_n} \langle \beta, C  \rangle = \min_{A\in \M_n(W)} g (A).
$$
Since $\M_n(W)\subset \M(W)$ the reverse inequality is immediate, and thus
$$
\inf_{A\in \M(W)} g (A) = \min_{A\in \M_n(W)} g(A).
$$ 
\end{proof}

\begin{proof}[Proof of Lemma~\ref{lem:Lj_constants_are_increasing}]

First, remark that if $W$ and $W'$ are two observation windows such that $W\subset W'$,
then, for all $y\in\R^d$ and $A\in\M$,
$$
0\leq \delta_{0;W}(A) - \delta_{y;W}(A) \leq \delta_{0;W'}(A) - \delta_{y;W'}(A).
$$
Indeed,   $W\subset W'$ yields $(A\setminus(y+A))\cap W\subset (A\setminus(y+A))\cap W'$,
and thus, taking the Lebesgue measure and using (\ref{eq:difference_zero_local_covariogram_is_leb_measure}) gives
$0\leq \delta_{0;W}(A) - \delta_{y;W}(A) \leq \delta_{0;W'}(A) - \delta_{y;W'}(A)$.

Let $W$ and $W'\in\W$ such that $W\subset W'$, $j\in\{1,\dots,d\}$, and let us show that $L_j(\gamma;W)\leq L_j(\gamma,W')$.
Suppose that $L_j(\gamma,W')$ is finite, otherwise there is nothing to show.
Since $W\subset W'$ one also has
$W\ominus[-\eps e_j,0]\subset W'\ominus[-\eps e_j,0]$
and
$W\ominus[0,\eps e_j]\subset W'\ominus[0,\eps e_j]$.
Hence, according to the preliminary remark,
for all $A\in\M$, $\eps\in\R$,
 $
\sigma _{\eps e_{j};W}(A) 
\leq \sigma _{\eps e_{j};W'}(A). 
 $
Since $\gamma$ is admissible, this implies that for all $\eps\in\R$,
$ 
\sigma_{\gamma } ( {\eps e_{j};W}) 
\leq 
\sigma_{\gamma } ({\eps e_{j};W'}) .
 $
Hence, by definition of $L_j(\gamma,W')$, 
$
\sigma_{\gamma } ( {\eps e_{j};W}) 
\leq L_j(\gamma,W'),
$
and thus $L_j(\gamma;W) \leq L_j(\gamma,W')$.
\end{proof}

\begin{proof}[Proof of Lemma~\ref{lem_admissible_Delta_are_locally_lipschitz}]
Recall that it has been shown in the proof of Proposition~\ref{prop:specific_covariogram_directional_lipschitz_constant} (see~\eqref{eq:difference_local_covariogram_less_than_in_zero_translated})
that for all $y, z\in\R^d$, $W\in\W$, and $A\in\M$,
$$
\delta_{y;W}(A) - \delta_{z,W}(A)
\leq \delta_{0,-y+W}(A) - \delta_{z-y,-y+W}(A).
$$
Let $\gamma$ be an admissible function and let $y\in\R^d$, $W\in\W$ and $r>0$ be fixed.
Let $z, z'\in C(y,r)$ be such that $z' = z + t e_j$ for some $t\in\R$ and $j\in\{1,\dots,d\}$.
Since $\gamma$ is admissible, the above inequality ensures that
$$
\gamma(z;W) - \gamma(z';W)
\leq \gamma(0;-z+W) - \gamma(t e_j; -z+W).
$$
As a consequence of (\ref{eq:difference_zero_local_covariogram_is_leb_measure}),
since $\gamma$ is admissible, the difference at zero $U\mapsto \gamma(0;U) - \gamma(t e_j; U)$ is an increasing function of $U$.
Hence, since $W \subset (W\oplus [-t e_j, 0]) \ominus[-t e_j,0]$, 
$$
\begin{aligned}
& \gamma(z;W) - \gamma(z';W)
\\
&\leq \gamma(0;-z+(W\oplus [-t e_j, 0]) \ominus[-t e_j,0]) - \gamma(t e_j; -z+(W\oplus [-t e_j, 0]) \ominus[-t e_j,0])
\\
& \leq |t| \frac{\gamma(0;-z+(W\oplus [-t e_j, 0]) \ominus[-t e_j,0]) - \gamma(t e_j; -z+(W\oplus [-t e_j, 0]) \ominus[-t e_j,0])}{|t|}
\\
& \leq |t| \sup_{\eps\in\R} \frac{\gamma(0;-z+(W\oplus [-t e_j, 0]) \ominus[-\eps e_j,0]) - \gamma(\eps e_j; -z+(W\oplus [-t e_j, 0]) \ominus[-\eps e_j,0])}{|\eps|}
\\
& \leq |t| L_j(\gamma, -z+W\oplus [-t e_j, 0]).
\end{aligned}
$$
According to Lemma~\ref{lem:Lj_constants_are_increasing}, $W\mapsto L_j(\gamma;W)$ is increasing.
Since $z\in C(y,r)$ and $|t| = \| z-z'\|_\infty \leq \|z-y\|_\infty+\|y-z'\|_\infty \leq 2r$,
we have $-z+W\oplus [-t e_j, 0] \subset W\oplus C(-y,3r)$.
Hence, for all $z, z'\in C(y,r)$ be such that $z' = z + t e_j$,
$$
\gamma(z;W) - \gamma(z';W) \leq  L_j(\gamma, W\oplus C(-y,3r)) |t|.
$$
Exchanging $z$ and $z'$ one gets, $|\gamma(z;W) - \gamma(z';W)| \leq L_j(\gamma, W\oplus C(-y,3r)) |t|$.
To finish,   consider a couple of points $z,z'\in C(y,r)$ that are not necessarily aligned along an axis.
Consider the finite sequence of vector $u_0=z$, $u_1$,$\dots$,$u_d=z'$ defined such that the $j$ first coordinates of $u_j$ are the ones of $z'$ while its $d-j$ last coordinates are the ones of $z$, so that $u_0 = z$, $u_d = z'$ and
$u_{j} - u_{j-1} = (z_j' - z_j) e_j$.
Clearly, each $u_j$ belongs to the hypercube $C(y,r)$, and thus applying the $d$ inequalities obtained above,
$$
\begin{aligned}
| \gamma(z;W) - \gamma(z';W) |
& = \left| \sum_{j=1}^d \gamma(u_j;W) - \gamma(u_{j-1};W) \right|
\\
& \leq \sum_{j=1}^d \left| \gamma(u_j;W) - \gamma(u_{j-1};W) \right|
\\
& \leq \sum_{j=1}^d L_j(\gamma, W\oplus C(-y,3r)) |z_j' - z_j|.
\end{aligned}
$$
If $\gamma$ satisfies~\eqref{eq:avg-perimeter} then, for all $n\in\N^*$ and $j\in\{1,\dots,d\}$, the constants $L_j(\gamma, U_n)$ are all finite.
According to Lemma~\ref{lem:Lj_constants_are_increasing}, this implies that the $d$ constants $L_j(\gamma, W\oplus C(-y,3r))$, $j\in\{1,\dots,d\}$, are all finite for any fixed $y\in\R^d$, $W\in\W$, and $r>0$, and thus the map $y\mapsto \gamma(y;W)$ is locally Lipschitz.
\end{proof}

\subsection{Stationary case}
\label{subsubsec:proof_realisability_Rd_stationary}

The following theorem is the main result of this paper.
It is a refined version of Theorem~\ref{thm:intro-result} given in the introduction.

\begin{theorem}
\label{thm:stationary-realisability_over_Rd_perB}
Let $S_2:\mathbb{R}^{d}\rightarrow \R$ be a  function and $r\geq 0$. Then there is a stationary RAMS $X$ such that 
\begin{equation}
\label{eq:beta-real-problem}
\begin{cases}
S_2(y) = \gamma_{X}^{s}(y), & y\in\R^d,\\
\PerB^{s}(X)\leq r
\end{cases}
\end{equation}
if and only if $S_2$ is admissible and
$$
\sum_{j=1}^d \Lip_j( S_2,0) \leq \frac{r}{2}.
$$
\end{theorem}
We shall use a variant of Theorem 2.10(ii) from \cite{LacMol}, where the monotonicity assumption is replaced by a domination.
 
\begin{theorem}
\label{thm:markov-kakutani}
Let $G$, $\chi$, $\Phi$ be like in Theorem~\ref{thm_lacmol_realisability}, and assume that $G$ is stable under the action of a group of transformations $\Theta $ of $\mathbb{R}^{d}$: For all $\theta \in \Theta$, $g\in G$, $\theta  g: A\mapsto g(\theta A)$ is a function of $G$.  
Assume furthermore that  there is a sequence 
$ (g_{n})_{n\geq 1}$ of functions of $G$ such that $0\leq g_{n}\leq \chi $ and
\begin{equation*}
g_{n}(A) \longrightarrow \chi (A)~\text{as}~n\to+\infty,\quad A\in \M,
\end{equation*}
and that $\chi $ is \emph{sub-invariant}: For every $\theta \in \Theta $, there is a constant $C_{\theta }>0$ such that
\begin{equation}
\label{eq:sub-stationarity}
\chi (\theta A )\leq C_{\theta }\chi (A),\quad A\in \M.
\end{equation} 
Then, if $\Phi $ is invariant under the action of $\Theta$, that is,
\begin{equation*}
\Phi (\theta g)=\Phi (g),\quad  g\in G,~\theta \in \Theta ,
\end{equation*}
for any given $r\geq 0$, 
there exists a $\Theta$-invariant RAMS $X$ such that 
$$
\begin{cases}
\E g(X) = \Phi (g),\quad g\in G,\\
\E \chi (X)\leq r
\end{cases}
$$
if and only if \eqref{eq:sup-inf-condition} holds.
\end{theorem}

\begin{proof}The proof is  the same as the one of Theorem 2.10 (ii) from \cite{LacMol}, itself based on Proposition 4.1 from Kuna \textit{et al.}~\cite{KunLebSpe11}. 
It consists in checking hypotheses of the Markov-Kakutani fixed point theorem. Let $M$ be the family of random elements $X$ that realise $\Phi $ on $G$, and satisfy 
 $
\E \chi (\theta X)\leq r$
  for every $\theta \in \mathbb{R}^{d}$. The family $M$ is easily seen to be convex with respect to addition of measures, it is compact by Theorem 2.8 from \cite{LacMol}, and invariant under the action of  $\Theta $ thanks to the $\Theta$-invariance of $\Phi$. It remains to prove that $M$ is not empty.
Since \eqref{eq:sup-inf-condition} is in order, Theorem~\ref{thm_lacmol_realisability} yields the existence of a RAMS $X$ realising $\Phi $ and such that $\E \chi (X)\leq r$.  For $\theta \in \Theta $, by Lebesgue theorem,  
\begin{equation*}
\E \chi (\theta X )=\E \lim_{n}g_{n}(\theta X )=\lim_{n}\E g_{n}(\theta X )=\lim_{n}\E g_{n}(X)=\E\lim_{n}g_{n}(X)=\E \chi (X)\leq r,
\end{equation*}where we have used the fact that 
$ \E \sup_{n}g_{n}(X)<\infty $
and 
\begin{equation*}
\E \sup_{n} g_{n}(\theta X )\leq \E \chi (\theta X )\leq C_{\theta }\E \chi (X)<\infty.
\end{equation*}
It follows that $M\ni X$ is non-empty, whence by the Markov-Kakutani theorem the mappings $X\mapsto \theta X$, $\theta \in \Theta $, admit a common fixed point $X$ (considered here as a probability measure), which is therefore invariant under $\Theta $. 
\end{proof}
 
\begin{proof}[Proof of Theorem~\ref{thm:stationary-realisability_over_Rd_perB}]
\textsl{Necessity:}
Assume $S_{2}$ is the specific covariogram of a stationary RAMS $X$ with $\PerB^{s}(X)\leq r$. 
Then $S_{2}$ is admissible, and by Proposition~\ref{prop:specific_covariogram_directional_lipschitz_constant}, 
$$
\sum_{j=1}^d \Lip_j( S_2,0)
= \sum_{j=1}^d \frac{1}{2} V^s_{e_j}(X)
= \frac{1}{2} \PerB^s(X)
\leq \frac{r}{2}.
$$

\noindent\textsl{Sufficiency:}
Define $\gamma (y;W) = \leb(W) S_2(y)$.
Then, for all $W\in\W$ and $j\in\{1,\dots,d\}$,
$$
\begin{aligned}
& L_j(\gamma, W)\\
& = \sup_{\eps \in \mathbb{R}}\frac{1}{ | \eps  | }\left[\L^d(W\ominus[-\eps e_j,0])(S_2(0)-S_2(\eps e_j)) + \L^d(W\ominus[0,\eps e_j])(S_2(0) - S_2(-\eps e_j) \right]
\\
& \leq 2 \L^d(W) \Lip_{ {j}}(S_{2},0).
\end{aligned}
$$
Hence, since $\sum_{j=1}^{d}\Lip_{j}(S_{2},0)\leq \frac{r}{2}$, 
$$
\sum_{n\geq 1}\beta_{n}\left(\sum_{j=1}^d L_j (\gamma, U_{n})\right)
\leq \sum_{n\geq 1}\beta_{n}  2 \L^d(U_n)\sum_{j=1}^{d}\Lip_{j}(S_{2}) \leq r.
$$
Hence, according to Theorem~\ref{thm:realisability_over_Rd_perBbeta},
$\gamma$ is the local covariogram of a RAMS $X$, 
and consequently, according to Theorem~\ref{thm_lacmol_realisability}, 
$\gamma$ satisfies \eqref{eq:sup-inf-condition} with $\chi = \PerB^\beta$ and $ G$ the space of all functionals of the form $g=\sum_{i=1}^{q}a_{i}\delta _{y_{i};W_{i}}$.
For $t\in \mathbb{R}^{d}, y\in \mathbb{R}^{d}, W\in \W, A\in \M$, we have
$$
\begin{aligned}
\theta_{t} \delta_{y;W} (A) = \delta _{y;W}(\theta_{t}A)=\leb((t+A)\cap (t+y+A)\cap W) & = \leb(A\cap (A+y)\cap (-t+W))\\
& =\delta _{y;-t+W}(A),
\end{aligned}
$$
whence the space $G$ generated by constant functions and functions $\delta _{y;W},y\in \mathbb{R}^{d},W\in \W$ is invariant under the action of the group $\Theta = \{\theta_t:~t\in\R^d\}$ of  translations. The linear functional defined by 
\begin{equation*}
\Phi (\delta _{y;W})=\leb(W)\gamma^{s}(y)
\end{equation*}
is invariant under the action of translations $\theta _{t}, t\in \mathbb{R}^{d}$.
For $t\in \mathbb{R}^{d}$,
let $\lceil t \rceil_\infty  = \lceil \|t\|_\infty \rceil$ be the smallest integer larger than $\|t\|_\infty$. 
Then, recalling that $U_n$ denotes the hypercube $(-n,n)^d$,  
$-t + U_{n} \subset (-n-\|t\|_\infty,n+\|t\|_\infty)^d \subset U_{n+\lceil t \rceil_\infty}$.
Hence, for $A\in \M$, 
$$
\PerB^\beta(t+A)
= \sum_{n=1}^{+\infty}\beta_{n}\PerB(t+A;U_{n}) 
= \sum_{n=1}^{+\infty}\beta_{n} \PerB(A; -t + U_{n})
\leq \sum_{n=1}^{+\infty} \beta_{n} \PerB(A; U_{n+\lceil t \rceil_\infty}).
$$
Since $\beta_n = 2^{-n}(2n)^{-d} = 2^{\lceil t \rceil_\infty}\left( \frac{n+\lceil t \rceil_\infty}{n} \right)^{d}\beta _{n+\lceil t \rceil_\infty} \leq 2^{\lceil t \rceil_\infty}\lceil t \rceil_\infty^{d} \beta _{n+\lceil t \rceil_\infty}$, we have
$$
\PerB^\beta(t+A) \leq 2^{\lceil t \rceil_\infty}\lceil t \rceil_\infty^{d} \sum_{n=1}^{+\infty} \beta _{n+\lceil t \rceil_\infty} \PerB(A; U_{n+\lceil t \rceil_\infty})
\leq 2^{\lceil t \rceil_\infty}\lceil t \rceil_\infty^{d} \PerB^\beta(A),\quad A\in\M,
$$
whence~\eqref{eq:sub-stationarity} is in order for $\chi = \PerB^\beta$.
To apply Theorem~\ref{thm:markov-kakutani} it only remains to check that $\PerB^\beta$ can be pointwise approximated from below by functions from $G$.
According to Proposition~\ref{prop:per-approx}, for $A\in \M$, and $U$ a bounded  open set   of $\mathbb{R}^{d}$,  
\begin{equation*}
\PerB(A;U)=\lim_{n}g^{U}_{n }(A) 
\end{equation*}
for some function $g_{n}^{U}\in G$ mentioned in Proposition~\ref{prop:per-approx}
that satisfy 
\begin{equation}
\label{eq:intermed}
0\leq g_{n}^{U}\leq \PerB(\cdot ;U).
\end{equation}
Define 
\begin{equation*}
g_{n}(A)=\sum_{m=1}^{n}\beta _{m}g_{n}^{U_{m}}(A).
\end{equation*}
Let $A\in \M$ with $\PerB^{\beta }(A)<\infty $.
Since for every $m\geq 1$, $g_{n}^{U_{m}}(A)\to \Per(A;U_{m})$ as $n\to \infty $, the Lebesgue theorem with 
$0\leq g_{n}(A)\leq \PerB^{\beta }(A)<\infty$
ensures that $g_{n}(A)\to \PerB^{\beta }(A)$ as $n\to \infty $.

If $A\in \M$ is such that $\PerB^{\beta }(A)=\infty $, let $M>0$, and $n_{0}$ be such that $\sum_{m=1}^{n_{0}}\beta _{m}\PerB(A;U_{m})\geq M+1$. Let $n_{1}\geq n_{0}$ be such that for $n\geq n_{1}$, $g_{n}^{U_{m}}(A)\geq \PerB(A;U_{m})-1$ for $1\leq m\leq n_{0}$. Then, for $n\geq n_{1}$, 
$$
\begin{aligned}
g_{n}(A)\geq \sum_{m=1}^{n_{0}}\beta _{m}\PerB(A;U_{m})-\sum_{m\geq 1}\beta _{m}\geq M+1-1\geq M.
\end{aligned}
$$It follows that $g_{n}(A)\to \infty =\PerB^{\beta }(A)$.

Hence, according to Theorem~\ref{thm:markov-kakutani}, there exists a stationary RAMS $X$ realising $\gamma$, which implies that $\gamma^s_X = S_2$.
Then, according to Proposition~\ref{prop:specific_covariogram_directional_lipschitz_constant}, 
$$
\PerB^{s}(X) = \sum_{j=1}^d V_{e_j}^s (X) = 2 \sum_{j=1}^d \Lip_j(S_2,0) \leq r.
$$
\end{proof}

\subsection{Covariogram realisability problem for RACS of $\mathbb{R}$}

The goal of this section is to establish a result similar to Theorem~\ref{thm:stationary-realisability_over_Rd_perB} for the specific covariogram of one-dimensional stationary RACS.

First let us discuss the definition of local covariogram admissibility of functions in arbitrary dimension $d\geq 1$.
By analogy with the definition of $\M$-local covariogram admissible functions (see Definition~\ref{def:covariogram_admissible_measurable}), 
when considering RACS of $\R^d$,  
one says that a function $\gamma:\R^d\times \W \rightarrow \R$ is \emph{$\F$-local covariogram admissible} if
for all 5-tuples $(q\geq 1, (a_i)\in \mathbb{R}^{q}  , (y_i)\in (\mathbb{R}^{d})^{q} , (W_i)\in \W^q, c\in  \mathbb{R})$,
$$
\left[\forall F\in \F,\quad c+\sum_{i=1}^{q}a_{i}\delta_{y_{i};W_i}(F)\geq 0\right]
\Rightarrow
c+\sum_{i=1}^{q}a_{i}\gamma(y_i;W_i)\geq 0.
$$
Besides, one says that $S_2:\R\rightarrow\R$ is \emph{$\F$-specific covariogram admissible} if $(y,W)\mapsto S_2(y) \L^d(W)$ is 
$\F$-local covariogram admissible.
However, this distinction is superfluous since these two notions of admissibility are strictly equivalent.

\begin{proposition}
\label{prop:equivalence_positiveness_closed_measurable}
A function $\gamma:\R^d\times \W \rightarrow \R$ is $\F$-local covariogram admissible if and only if it is $\M$-local covariogram admissible.
\end{proposition}

\begin{proof}[Proof of Proposition~\ref{prop:equivalence_positiveness_closed_measurable}]
The proof of this equivalence relies on the continuity of local covariograms for the convergence in measure and the density of compact sets due to the Lusin theorem.
It consists in showing that for all 5-tuples
$(q\geq 1, (a_i)\in \mathbb{R}^{q}  , (y_i)\in (\mathbb{R}^{d})^{q} , (W_i)\in \W^q, c\in  \mathbb{R})$,
$$
\left[\forall F\in \F,\quad c+\sum_{i=1}^{q}a_{i}\delta_{y_{i};W_i}(F)\geq 0\right]
\Leftrightarrow
\left[\forall A\in \M,\quad c+\sum_{i=1}^{q}a_{i}\delta_{y_{i}; W_i}(A)\geq 0\right].
$$
Since $\F\subset \M$, the implication $\Leftarrow$ is clear.
To show the converse, let $(q\geq 1, (a_i)\in \mathbb{R}^{q}  , (y_i)\in (\mathbb{R}^{d})^{q} , (W_i)\in \W^q, c\in  \mathbb{R})$ be such that $\forall F\in \F,\quad c+\sum_{i=1}^{q}a_{i}\delta_{y_{i}}(F)\geq 0$, and let us show that this inequality is valid for any $A\in \M$.
One can suppose that $A$ is bounded since according to~\eqref{eq:locality_local_covariogram},
one can replace $A$ by 
$A \cap \bigcup_{i=1}^q W_i\cup(-y_i+W_i)$
without changing the value of $c+\sum_{i=1}^{q}a_{i}\delta_{y_{i}; W_i}(A)$.
Then, by the Lusin theorem (see e.g.~\cite{Evans_Gariepy_measure_theory_and_fine_properties_of_functions_1992}), there exists a sequence of compact sets $K_n\subset A$ that converges in measure towards $A$ ,\emph{i.e.} $\L^d(A\Delta K_n)\rightarrow 0$.
Since each $K_n$ is closed, for all $n$, $c+\sum_{i=1}^{q}a_{i}\delta_{y_{i}}(K_n)\geq 0$.
Since the sequence $(K_n)_{n\in\N}$ converges in measure towards $A$, 
thanks to the fact that the local covariogram $B\mapsto \delta_{y;W}(B)$ is continuous for the local convergence in measure (see Proposition~\ref{prop:continuity_local_covariogram})
for all $(y_i;W_i)$, $\delta_{y_i;W_i}(K_n)$ tends to $\delta_{y_i;W_i}(A)$, and thus letting $n$ tend to $+\infty$ gives the inequality $c+\sum_{i=1}^{q}a_{i}\delta_{y_{i};W_i}(A)\geq 0$.
\end{proof}

Now that this technical point has been clarified we are in position to formulate our result for the realisability of specific covariogram of stationary RACS of $\R$.

\begin{theorem}
\label{thm:realisability_for_1dRACS}
Suppose that the probability space $(\Omega, \mathcal{A},\P)$ is complete.
Let $S_2 : \R\rightarrow\R$ be a given function and let $r>0$.
Then $S_2$ is the covariogram of a stationary RACS $Z\subset\R$ such that
$$
\E \left(\H^0(\partial Z)\cap (0,1)\right) \leq r
$$
if and only if $S_2$ is $\F$-specific covariogram admissible and Lipschitz with Lipschitz constant $L \leq \frac{r}{2}$.
\end{theorem}

\begin{proof}
\textsl{Necessity:}
If there exists a stationary RACS $Z\subset \R$ such that $\E \left(\H^0(\partial Z)\cap (0,1)\right) \leq r$, then $S_2$ is necessarily $\F$-specific covariogram admissible and, according to Proposition~\ref{prop:specific_covariogram_directional_lipschitz_constant}, $S_2$ is Lipschitz with Lipschitz constant
$L = \frac{1}{2} \E\left(\Per(Z);(0,1)\right)$.
But  $\Per(Z;(0,1))\leq \H^0(\partial Z\cap (0,1))$ yields $L \leq \frac{1}{2} \E \left(\H^0(\partial Z\cap (0,1) )\right) \leq  \frac{r}{2}$.

\textsl{Sufficiency:}
Suppose that $S_2$ is $\F$-specific covariogram admissible and Lipschitz with Lipschitz constant $L \leq \frac{r}{2}$.
Then, by Proposition~\ref{prop:equivalence_positiveness_closed_measurable}, $\gamma$ is $\M$-specific covariogram admissible, and thus by Theorem~\ref{thm:stationary-realisability_over_Rd_perB}
there exists a RAMS $X\subset\R$ such that $S_2$ is the specific covariogram of $X$ and $\E(\Per(X);(0,1))\leq r$.
By Proposition~\ref{prop_1d_racs_representative}, there exists a RACS $Z\subset \R$, equivalent in measure to $X$, such that $\Per(X;(0,1)) = \H^0(\partial Z\cap (0,1))$ a.s.
Then the specific covariogram of $Z$ is also equal to $S_2$ and $\E \left(\H^0(\partial Z\cap (0,1))\right) = \E(\Per(X;(0,1)))\leq r$.
\end{proof}

Note that although the geometry of sets with finite perimeter on the line seems quite simplistic, a direct proof of the realisability result above is far from trivial.

\appendix

\section{The Radon-Nikodym theorem for random measures}

\begin{theorem}
\label{thm_radon_nikodym_random_measure_ac_lebesgue}
Let $U$ be an open subset of $\R^d$.
Let $\mu:\Omega \mapsto \mathbf{M}(U)$ be a random signed Radon measure on $U$ such that for all $\omega\in\Omega$,
the measure $\mu(\omega,\cdot)$ is absolutely continuous with respect to the Lebesgue measure $\leb$.
Then, there exists a jointly measurable map $f$ on
$
 \left(\Omega\times\R^d, \mathcal{A} \otimes \mathcal{B}\left(U\right)\right) 
 $ such that for all $\omega\in\Omega$, $f(\omega,\cdot)$ is a Radon-Nikodym derivative of $\mu(\omega,\cdot)$ with respect to the Lebesgue measure $\leb$.\end{theorem}

\begin{proof}
This proof follows the outline of~\cite[Exercise 6.10.72]{Bogachev_measure_theoryII_2007}.
It is enough to consider the case $U=\R^d$, since for $U\subset\R^d$ one can always extend the random measure by zero over $\R^d$ and take the restriction of $f$ to $U$ afterwards.
Denote by $B(x,r)$ the open ball of center $x$ and radius $r$, and by $\kappa_d$ the Lebesgue measure of the unit ball of $\R^d$, so that for all $x\in \R^d$ and $r>0$, $\leb(B(x,r)) = \kappa_d r^d$.
For any $\omega\in\Omega$,
according to the Besicovitch derivation theorem (see e.g.~\cite[Theorem 2.22]{Ambrosio_Fusco_Pallara_functions_of_bounded_variation_free_discontinuity_problems_2000}),
the derivative of the measure $\mu(\omega,\cdot)$ with respect to $\leb$, that is
$$
\lim_{r\rightarrow 0+} \frac{\mu(\omega,B(x,r))}{\kappa_d r^d},\quad x\in \R^d,
$$
exists for $\leb$-almost all $x\in \R^d$, is in $L^1(\R^d)$, and is a Radon-Nikodym derivative of the measure $\mu(\omega,\cdot)$.
Let $(r_n)_{n\in\N}$ be a positive sequence decreasing to $0$.
For all $\omega\in\Omega$, $x\in \R^d$, and $n\in\N$, define
$$
f_n(\omega,x) = \frac{\mu(\omega,B(x,r_n))}{\kappa_d r_n^d}.
$$
As a consequence of the Besicovitch derivation theorem, for all $\omega\in \Omega$,
the function
$$
f(x,\omega) = \limsup_{n\rightarrow +\infty}f_n(\omega,x) \Ind\left(\limsup_{n\rightarrow +\infty}f_n(\omega,x) = \liminf_{n\rightarrow +\infty}f_n(\omega,x)\right)
$$
is a Radon-Nikodym derivative of $\mu(\omega,\cdot)$ with respect to the Lebesgue measure $\leb$.
Let us show that this function $f$ is jointly measurable, \emph{i.e.} $\mathcal{A} \otimes \mathcal{B}\left(U\right)$-measurable.
Given the definition of $f$, and since the $\limsup$ and $\liminf$ of a countable sequence of measurable functions is a measurable function, it is enough to show that the functions $f_n$ are jointly measurable.
Let $n\in\N$.
For all $x\in\R^d$, by definition of a random Radon measure,
the map $\omega\mapsto f_n(\omega,x) = \frac{\mu(\omega,B(x,r_n))}{\kappa_d r_n^d}$ is $\mathcal{A}$-measurable.
Let us show that for all $\omega\in\Omega$, the map $x\mapsto f_n(\omega,x) = \frac{\mu(\omega,B(x,r_n))}{\kappa_d r_n^d}$
is continuous over $\R^d$.
Indeed, let $x\in\R^d$ and $(x_k)_{k\in\N}$ a sequence of points that tends to $x$.
Then, for all $y\in\R^d\setminus\partial B(x,r_n)$,
$\Ind(y\in B(x_k,r_n))$ tends to $\Ind(y\in B(x,r_n))$.
Since the sphere $\partial B(x,r_n)$ is Lebesgue negligible, by absolute continuity, $\partial B(x,r_n)$ is also $\mu(\omega,\cdot)$-negligible.
Hence, $\Ind(y\in B(x_k,r_n))$ tends to $\Ind(y\in B(x,r_n))$ for $\mu(\omega,\cdot)$-almost all $y\in\R^d$.
Besides, since the sequence $(x_k)$ tends to $x$, it is bounded, and thus there exists $R>0$ such that $x_k\in B(0,R)$ for all $k\in\N$.
Then, for all $k\in\N$,
$$
\left|\frac{\Ind(y\in B(x_k,r_n))}{\kappa_d r_n^d}\right|
\leq \frac{\Ind(y\in B(0,R+r_n))}{\kappa_d r_n^d} \in L^1(\mu(\omega,\cdot)).
$$
Hence, by dominated convergence,
$$
\lim_{k\rightarrow +\infty} \frac{\mu(\omega,B(x_k ,r_n))}{\kappa_d r_n^d} = \frac{\mu(\omega,B(x,r_n))}{\kappa_d r_n^d},
$$
that is $f_n(\omega,\cdot)$ is continuous at $x$.
In conclusion, $\omega\mapsto f_n(\omega,x)$ is measurable and $x\mapsto f_n(\omega,x)$ is continuous, \emph{i.e.} $f_n$ is a Carath\'eodory function.
Since $\R^d$ is a separable metric space, one can conclude that $f_n$ is jointly measurable~\cite[Section 4.10]{Aliprantis_Border_infinite_dimensional_analysis_2007}.
\end{proof}

%
%
%
%
\bibliographystyle{apt}

 \end{document}